\documentclass[10pt]{article}

\usepackage{amsfonts,amssymb,latexsym,hyperref}
\usepackage{amsthm}
\usepackage[margin=1in]{geometry}
\usepackage{graphicx}
\usepackage{amsmath}
\usepackage[font=footnotesize,labelfont=bf]{caption}
\usepackage{subfig}
\usepackage{placeins}
\usepackage{tikz}
\usepackage{marginnote}  
\usepackage{chngcntr}
\usetikzlibrary{calc,intersections}

\newcommand\R{\mathbb{R}}

\newcommand\Z{\mathbb{Z}}

\begin{document}

\numberwithin{figure}{section}
\newtheorem{thm1}{Theorem}[section]
\newtheorem{thm}[thm1]{Theorem}
\newtheorem{lem}[thm1]{Lemma}
\newtheorem*{thm2}{Theorem}
\newtheorem*{definition}{Definition}
\newtheorem{cor}[thm1]{Corollary}
\newtheorem{claim}[thm1]{Claim}
\newtheorem{conj}[thm1]{Conjecture}
\newtheorem*{rem}{Remark}
\newtheorem*{q}{Question}
\newtheorem*{question}{Question}
\newtheorem{ex}{Example}[section]

\newtheorem{corollary}[thm1]{Corollary}
\newtheorem{conjecture}[thm1]{Conjecture}
\newtheorem{theorem}[thm1]{Theorem}
\newtheorem{lemma}[thm1]{Lemma}
\newtheorem{exer}[thm1]{Exercise}
\newtheorem{proposition}[thm1]{Proposition}
\newtheorem{prop}[thm1]{Proposition}

\newcommand{\bi}{\begin{itemize}}
\newcommand{\ei}{\end{itemize}}
\newcommand{\be}{\begin{enumerate}}
\newcommand{\ee}{\end{enumerate}}
\newcommand{\ds}{\displaystyle}
\newcommand{\ul}{\underline}
\newcommand{\hnote}[1]{\marginnote{ \scriptsize \textcolor{red}{HNH:{#1}}}}
\newcommand{\anote}[1]{\marginnote{ \scriptsize \textcolor{orange}{Anna:{#1}}}}

    \counterwithout{figure}{section}
\interfootnotelinepenalty=10000

\title{An infinite family of knots whose hexagonal mosaic number is only realized in
              non-reduced diagrams}
\date{\today}
\author{Hugh Howards, Jiong Li, Xiaotian Liu, Anna Paulec}

\maketitle

\begin{abstract}
We give an infinite family of knots such that for any given $r \geq 3$, the family contains a knot which can be embedded on a hexagonal $r$-mosaic, but cannot fit on  a hexagonal $r$-mosaic in an embedding that achieves its crossing number.  This extends the rectangular mosaic result of Ludwig, Evans, and Paat \cite{L}.  
We also introduce a new tool for systematically finding all possible flypes for the diagram of any link thus making it easier to find all possible minimal crossing embeddings of prime, alternating knots.

\end{abstract}

\section{Introduction}

A great deal of work in topology and knot theory has been devoted to the understanding of quantum physics since Vaughn Jones' extraordinary creation of the Jones Polynomial \cite{J2}.  
Lomonaco and Kauffman introduced a way of studying quantum knots by defining knot mosaics in \cite{LK} .  They explain that mosaics give a blueprint to create an actual physical quantum system.  
Among other things they use it to glean information about the computational power necessary in order to simulate a quantum system.  Since then the study of mosaics and quantum knots has exploded.

Lomonaco and Kauffman continued exploration of mosaics in numerous papers including 
\cite{LK4} and \cite{LK3} as have many others.  An $r$-mosaic for Lomonaco and Kauffman's rectangular setting consists of an $r \times r$ board of square tiles with arcs of a knot or link running across them. 
More recently Jennifer McLoud-Mann expanded the idea by introducing hexagonal knot mosaics in \cite{jmm}.  Just as the plane can be tiled with squares it can be tiled with regular hexagons giving a new construction for quantum knots.  The possible hexagonal tiles up to rotation are depicted in Figure~\ref{fig:tiles}.  A bound for crossing numbers on hexagonal mosaics was given in \cite{hll}.  
A hexagonal $r$-mosaic consists of all the tiles in a hexagonal grid of radius $r$, so a 1-mosaic only contains one tile, while a 2-mosaic contains the central tile as well as the 6 tiles touching it etc.  We see the trefoil on a hexagonal 2-mosaic in Figure~\ref{fig:trefoil} on the left and a blank 4 mosaic board on the right.

\begin{figure}[ht]
	\centering
	{\includegraphics[width=.99\textwidth]{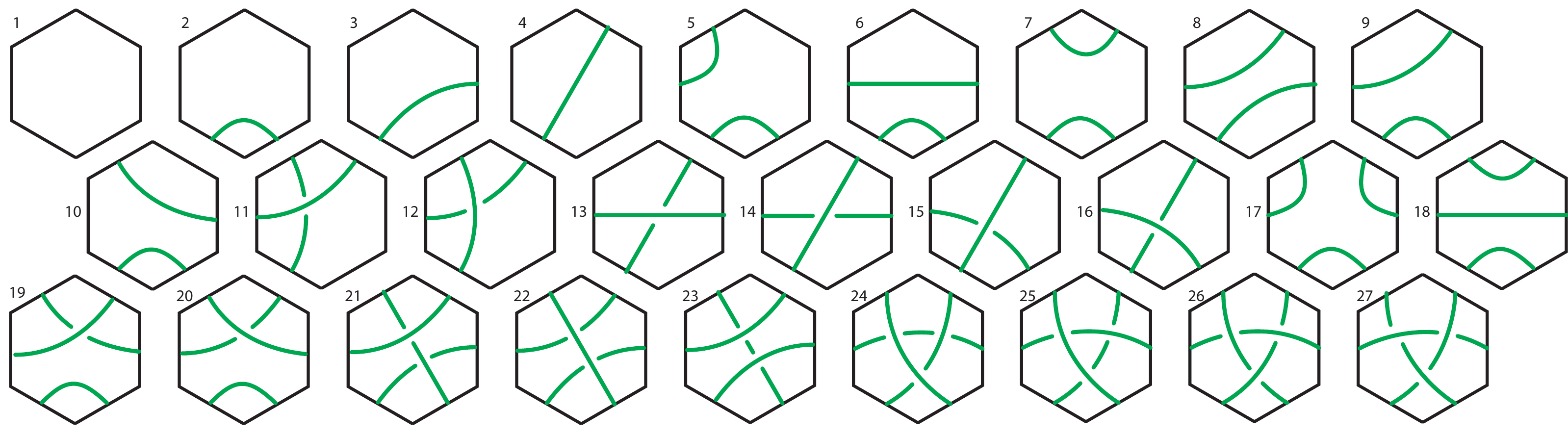}}
	\caption{The hexagonal mosaic tiles up to rotation.}
\label{fig:tiles}
\end{figure}

\begin{figure}[tpb]
\centering
\includegraphics[width=.25\textwidth]{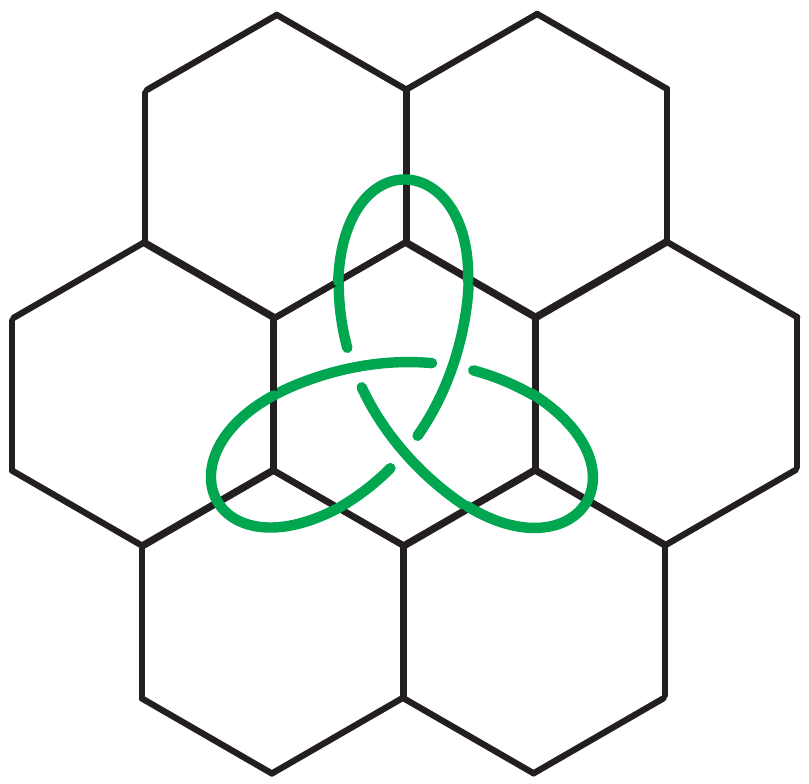}
\includegraphics[width=.27\textwidth]{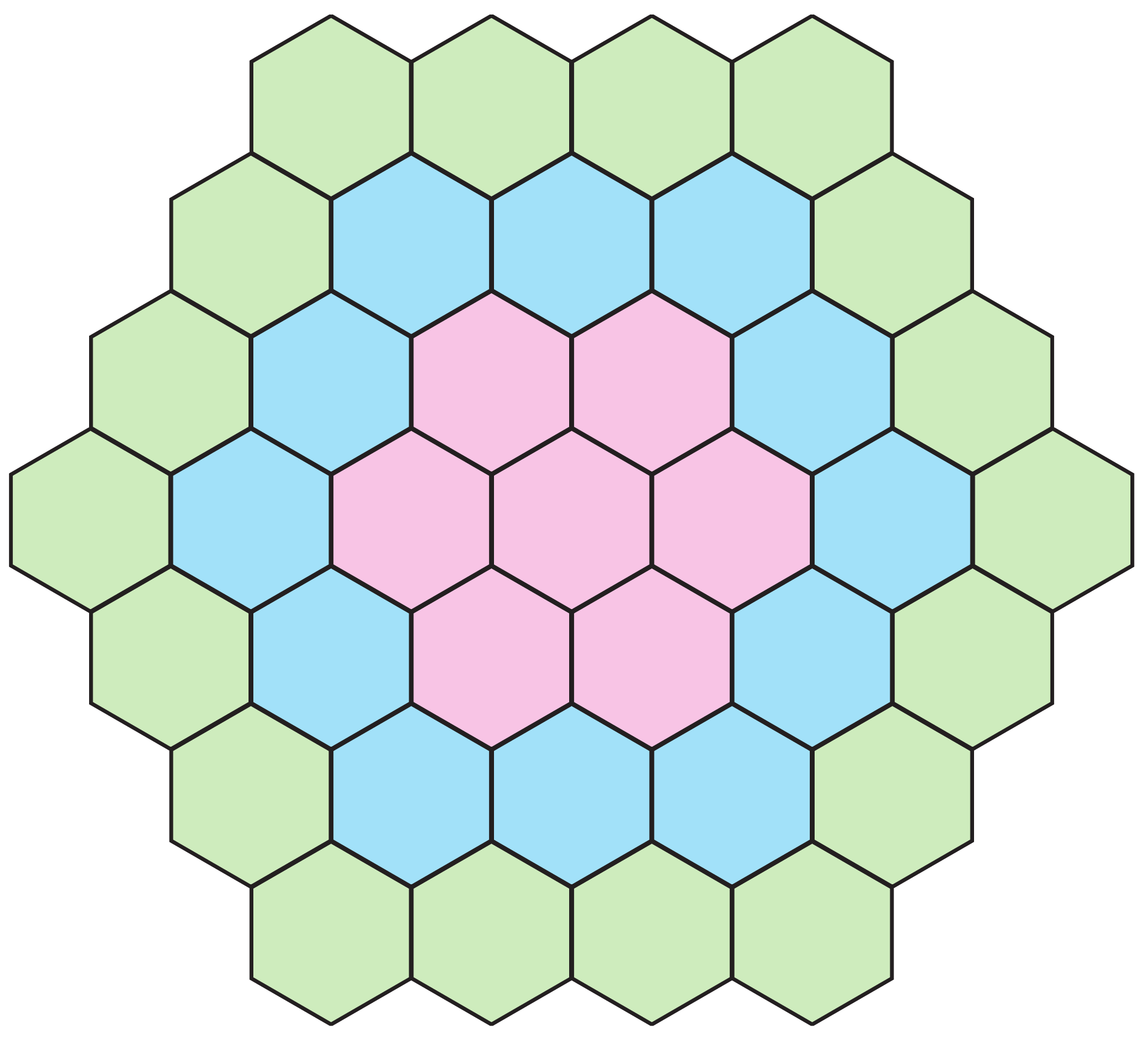}
\caption{On the left we see the trefoil on a hexagonal 2-mosaic.  On the right we see a 4-mosaic board.  The green tiles are the boundary tiles.  All the other tiles are called interior tiles.   The pink tiles which are a subset of the interior tiles are called the central tiles. The blue tiles are called the  penultimate corona (the second corona here).}

\label{fig:trefoil}
\label{fig:int}
\end{figure}

In this paper we extend a result of result of Ludwig, Evans, and Paat \cite{L} from rectangular mosaics to the hexagonal setting.
We give an infinite family of knots such that for any given $r \geq 3$, the family contains a knot which can be embedded on a hexagonal $r$-mosaic, but cannot fit on  a hexagonal $r$-mosaic in an embedding that achieves its crossing number.  Here we prove the result in the standard hexagonal setting, but the proof likely can be adapted to work in the semi-enhanced and enhanced setting as well (see Anna Paulec's master's thesis \cite{AP} for partial results in this direction as well as a full result in a new setting called crossing enhanced).  For the sake of brevity we restrict to the standard case here.

A second useful result in this paper that applies to knots in general and not just mosaics relates to finding all possible minimal crossing diagrams of prime, alternating knots or links. Menasco and Thistlethwaite's  solution \cite{MT} to Tait's famous $19^{th}$ century flyping conjecture  confirms that all
minimal crossing diagrams of a given prime, alternating knot or link are related by a sequence of flypes.
There are, however, multiple historical examples of flypes being overlooked in a knot diagram, including famously the Perko Pair (although in that case the knot was not alternating).   In this paper we introduce a new way of using a dual graph to systematically detect all possible flypes in a given diagram.  Note that while we apply the tool to learn information about mosaics it works just as well for any diagram of a link.

\section{Definitions}

\begin{figure}[!ht]
	\centering
	\setlength{\unitlength}{0.1\textwidth}
	\begin{picture}(10,3.0)

	\put(0,-.057){\includegraphics[width=.32\textwidth]{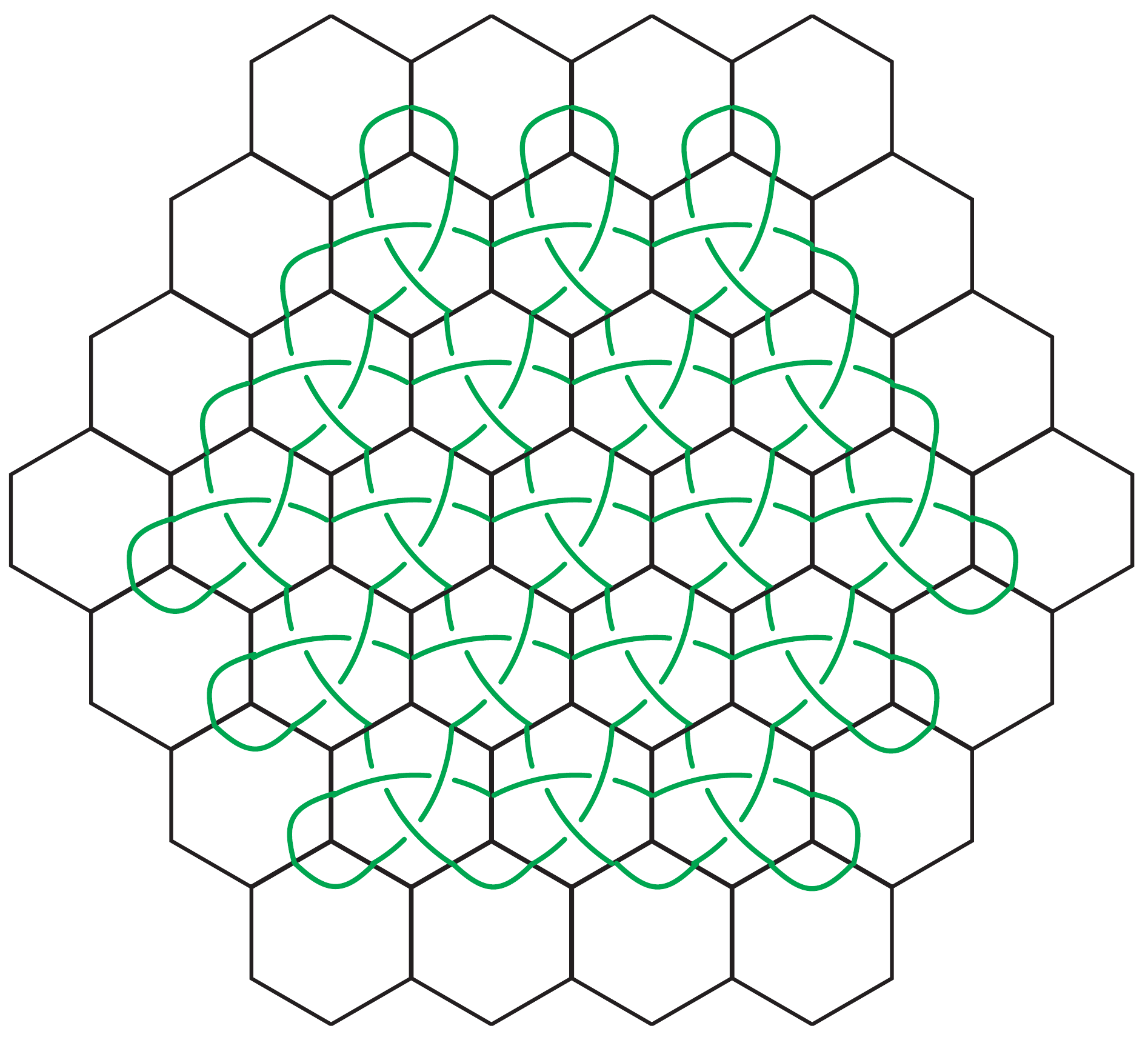}}
	\put(3.3,-.057){\includegraphics[width=.32\textwidth]{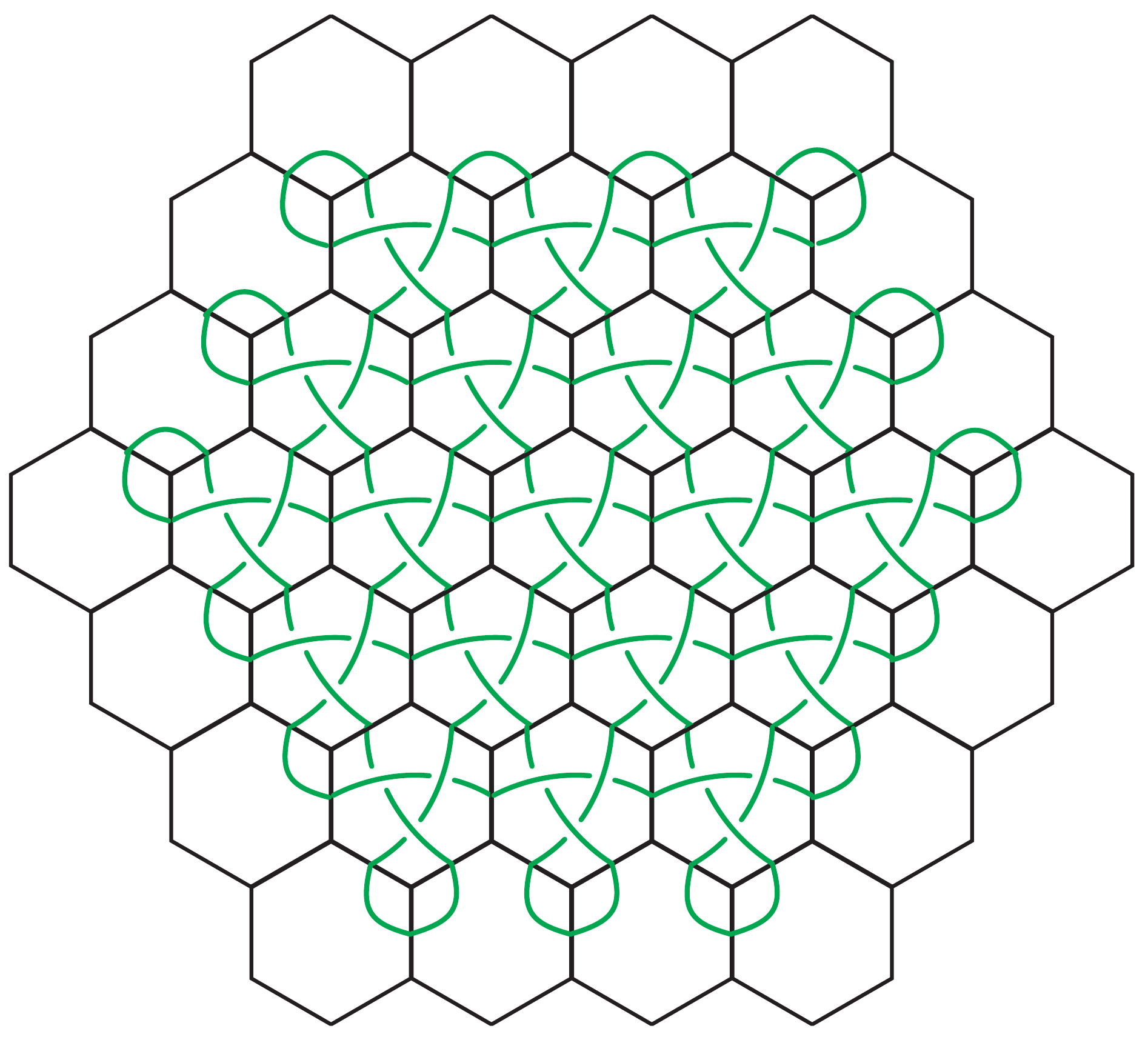}}
	\put(6.6,-.057){\includegraphics[width=.32\textwidth]{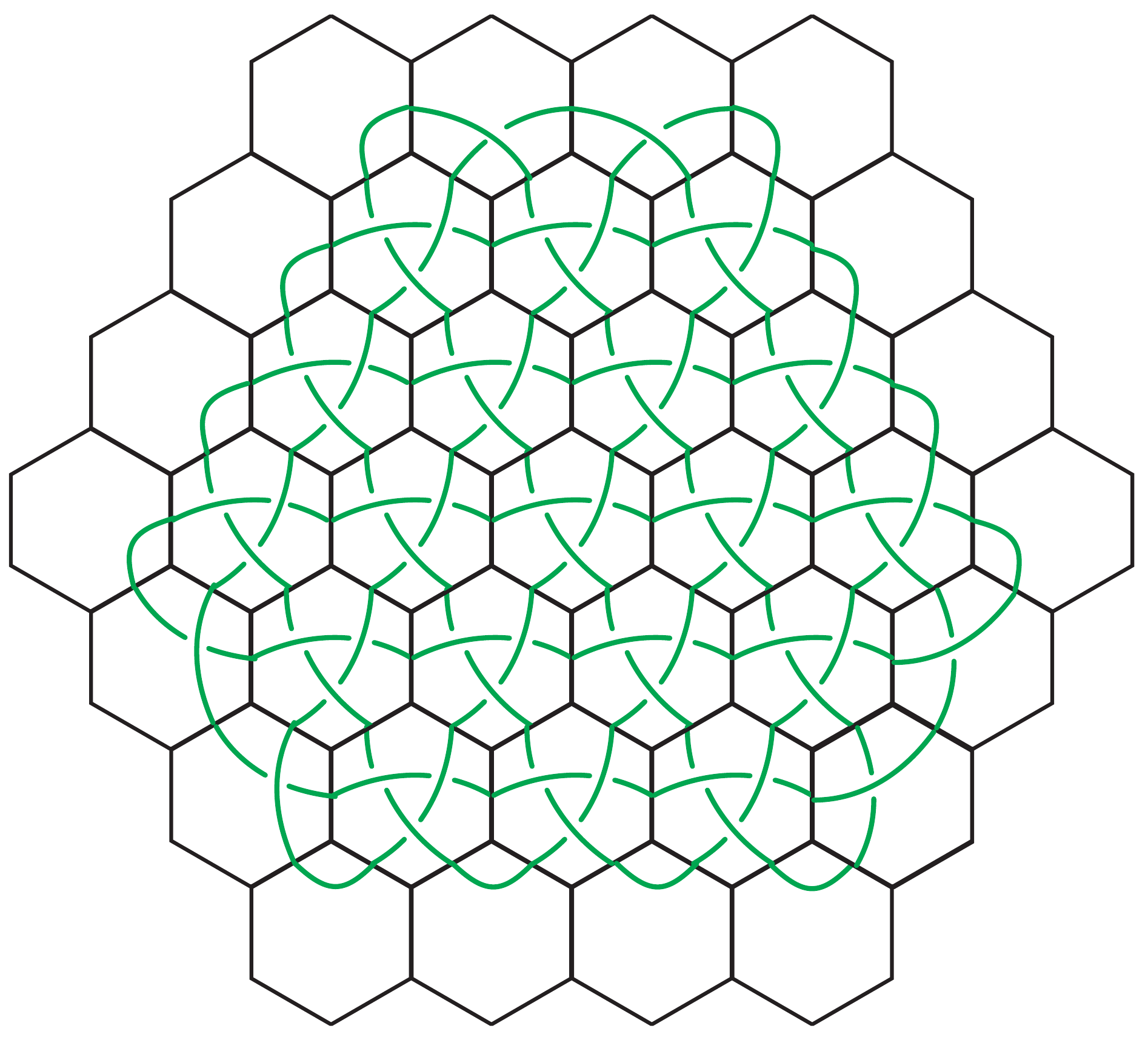}}
	\put(.3,2.5){$L_4$}
	\put(6.9,2.5){$\widehat{L_4}$ }
\end{picture}
	\caption{Here we see three 
    hexagonal mosaics which all agree on their interior tiles, but differ on the boundary tiles.  The first two are the two standard mosaics that can be obtained from those interior tiles and the link on the left is called $L_4$.    The link on the right is $\widehat{L_4}$, an enhanced hexagonal mosaic where crossings are allowed even in the boundary tiles.}
	\label{fig:l4}
\label{fig:twoways} 
 \end{figure}

We now establish some definitions, most of which are developed in \cite{hll} and \cite{jmm}.  For rectangular mosaics an $r$-mosaic is an $r \times r$ mosaic board made up of square tiles.  
A \emph{hexagonal $r$-mosaic} is a board of hexagonal tiles of radius $r$ centered around a given hexagon.  The hexagonal 1-mosaic has only one tile, and as with rectangular mosaics the 1-mosaic is too small to be useful.  The 2-mosaic (seen in Figure~\ref{fig:trefoil}) has the central tile plus six adjacent tiles for a total of seven tiles and so on.  A rectangular $r$-mosaic is in the shape of a square with outer sides of length $r$.  A hexagonal $r$-mosaic 
is roughly in the shape of a (wavy) hexagon where each of the outer sides 
has $r$ tiles. Call the tiles that correspond to the vertices of the hexagon \emph{corner tiles}.

As with rectangular mosaics the midpoint of each tile edge is called a {\em connection point} if it is
the endpoint of an arc drawn on that tile and a given tile
is {\em suitably connected}
if each of its connection points is identified with a connection
point of an adjacent tile. As with other mosaic papers, we are only interested in suitably connected mosaics.  

  The tiles $t$ tiles from from the central tile are called the \emph{$t^{th}$ corona}, thus the first ring of six tiles which are adjacent to the center tile form the first corona and the second ring (of twelve tiles) just outside those form the second corona, etc. 
Following the traditional definition in the rectangular case, 
we call the tiles in the outer ring (the $r-1^{st}$ corona) the \emph{boundary tiles}.    The remaining tiles of the mosaic are called  \emph{interior tiles}.  Sometimes we further divide the interior tiles into their own boundary and interior.  Let the subset of the interior tiles up through the $r-3^{rd}$ corona be called the \emph{central tiles}
 and  the $r-2^{nd}$ corona be called the \emph{penultimate corona} or \emph{penultimate tiles}. The interior tiles in Figure~\ref{fig:int} are shaded pink.  The second corona is the penultimate corona in that mosaic and is shaded blue.  The third corona is the boundary corona and is shaded green.   
Given a link mosaic the set of arcs that result from intersecting the link with the interior tiles are called the \emph{interior arcs} of the mosaic.

While it is impossible for 
 rectangular mosaics to have crossing tiles as boundary tiles and be suitably connected, a priori a suitably connected hexagonal mosaic could have a crossing boundary tile.  In a \emph{standard hexagonal mosaic}
one prohibits crossings in the boundary tiles and only allows tiles 1 through 5 from Figure~\ref{fig:tiles} on those tiles.  
This case is parallel to the rectangular case because once the interior tiles are fixed, there are exactly two ways to choose the boundary tiles and stay suitably connected.   If one allows crossings on the boundary tiles, the setting is called \emph{enhanced hexagonal mosaics}.   There is also a third setting that is less restrictive than the standard hexagonal mosaic setting and more restrictive than the enhanced setting where the boundary requirements are relaxed to allow Tile 6 of Figure~\ref{fig:tiles}, but does not allow crossings.  These are called \emph{semi-enhanced hexagonal mosaics}.  

McCloud Mann et al. prove their theorems in the standard setting in \cite{jmm}.  In \cite{hll} all three settings are addressed.  Here, like McCloud Mann, we will restrict ourselves to the standard setting for the sake of brevity.  For the rare time when we mention enhanced mosaics, we will denote them by $\widehat{M}$  and just use $M$ for a standard hexagonal mosaic (in \cite{hll} we dealt with all three hexagonal settings and the rectangular one so the notation was more complex). 
 An example of a mosaic interior together with the two possible ways to complete it as a standard hexagonal mosaic is seen on the left and in the middle in Figure~\ref{fig:twoways}.  On the right in that figure is an enhanced hexagonal mosaic with the same interior board.

On a standard hexagonal 2-mosaic there is only one crossing tile so the only non-trivial links we can get are a trefoil shown in Figure~\ref{fig:trefoil} and a Hopf link.  However, the existence of a non-trivial knot on a 2-mosaic demonstrates some of the potential power of hexagonal mosaics since rectangular mosaics up through 3-mosaics only yield unknots and even up through 5-mosaics the only knots one can get are the trefoil and figure eight knot,  but hexagonal 3-mosaics already yield knots up to crossing number 19 and hexagonal 5-mosaics contain knots with crossing number 108.  In \cite{AAGRK} Alanis et al. show that hexagonal mosaics are indeed more efficient at generating knots than rectangular mosaics.

A \emph{saturated} hexagonal mosaic  is one that uses only 3 crossing tiles  (one of the tiles labeled 23, 24, 25, and 26 in Figure~\ref{fig:tiles}) on the interior.
For a saturated mosaic, the penultimate tiles break down into three types of tiles.   There are, of course, six corner tiles defined above.  
 Let $T$ be a non-corner tile in the penultimate corona of a mosaic link $L$. The intersection of  $T$ with the boundary tiles contains two hexagonal edges and thus up to two connection points.  If $L$ uses both connection points as it does when the interior tiles are saturated then obviously one of two things must happen, either there is an arc of $L$ in the two boundary tiles adjacent to $T$ connecting those two points to each other (see, for example,  tiles $C', J'$, and  $M'$  in mosaic $L'$ on the bottom left in Figure~\ref{lp}) or the arcs of $L$ in the adjacent boundary tiles do not directly connect the two points to each other (see, for example, tiles $H'$ and $G'$ in mosaic $L' $ in the same figure).  If the connection points of $T$ are not connected directly to each other we say 
$T$ is a \emph{type I} tile.  If they are connected directly then $T$ is said to be a \emph{type II} tile.  
When the tiles in penultimate corona use all the connection points it shares with the boundary tiles such as when it
contains only 3 crossing tiles all of the tiles in the penultimate corona are exactly one of these three types and as we move around the the corona one will hit a corner tile, then a side of all type I tiles, then a corner tile then a side of all type II tiles, and so on alternating types as you pass each corner tile.

\section{Families of links and knots}
\label{sec:fam}
In this section we give the construction for three families of knots and links.  The first two were introduced in \cite{hll}.    We then use these two to build the desired family of knots for this paper. We prove it has special properties in Section~\ref{sec:lud}.
We first build $L_r$, a link which is saturated and achieves the highest crossing number possible of any link on a hexagonal $r$-mosaic, then alter it into $A_r$, a knot which is alternating and which  in \cite{hll} we show achieves the highest crossing number possible for a knot on a hexagonal $r$-mosaic. Finally we build $K_r$, a knot which fits on a hexagonal $r$-mosaic, but as we will show in Theorem~\ref{thm:fit} does not fit on a $r$-mosaic while in a diagram that achieves its crossing number.

   $L_r$ is a saturated mosaic in which every interior tile is chosen to be 
Tile 27 in Figure~\ref{fig:tiles} (if we chose Tile 26 every time we would simply get the mirror image). 
Once the interior tiles are set there are (as always) two ways to connect the boundary tiles.  One of these ways will result in a link with many nugatory crossings and one will have none.  The link with no nugatory crossings is called $L_r$.  $L_2$ is the trefoil (Figure~\ref{fig:trefoil}). See Figure~\ref{fig:l3} for $L_3$,  Figure~\ref{fig:l4} for $L_4$ and Figure~\ref{fig:l5k5} for $L_5$.
Note that   one could rotate any of the saturated tiles to get related, but slightly different saturated links.

\begin{figure}[tpb]  \centering
	\setlength{\unitlength}{0.1\textwidth}
	\begin{picture}(9.0,3.5)
	\put(0.0,0.0){\includegraphics[width=.97\textwidth]{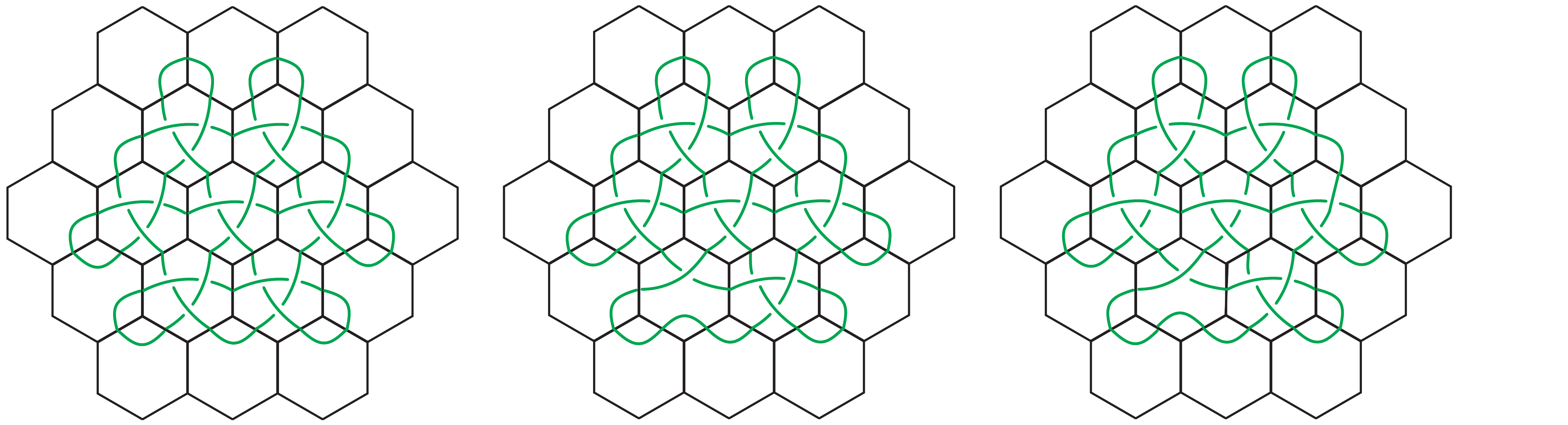}}
	\put(0.25,2.25){$L _3$}
	\put(3.3,2.25){$A_3$}
	\put(6.3,2.25){$K_3$}
	\end{picture}
    \caption{Here we see $L_3$ on the left, $A_3$ in the middle, and and $K_3$ on the right.  Because a hexagonal 3-mosaic only contains one central tile and that tile only intersects one component of $L_3$, $K_3$ is not built in the exact same way as the other examples where all the smoothings to from $K_r$ from $L_r$ can be done on central tiles($r \geq 4$).}
\label{fig:l3}
\label{fig:k3}
\end{figure}

\begin{figure}[tpb]  \centering
	\setlength{\unitlength}{0.1\textwidth}
	\begin{picture}(6.5,7.0)
	\put(0,-.057){\includegraphics[width=.37\textwidth]{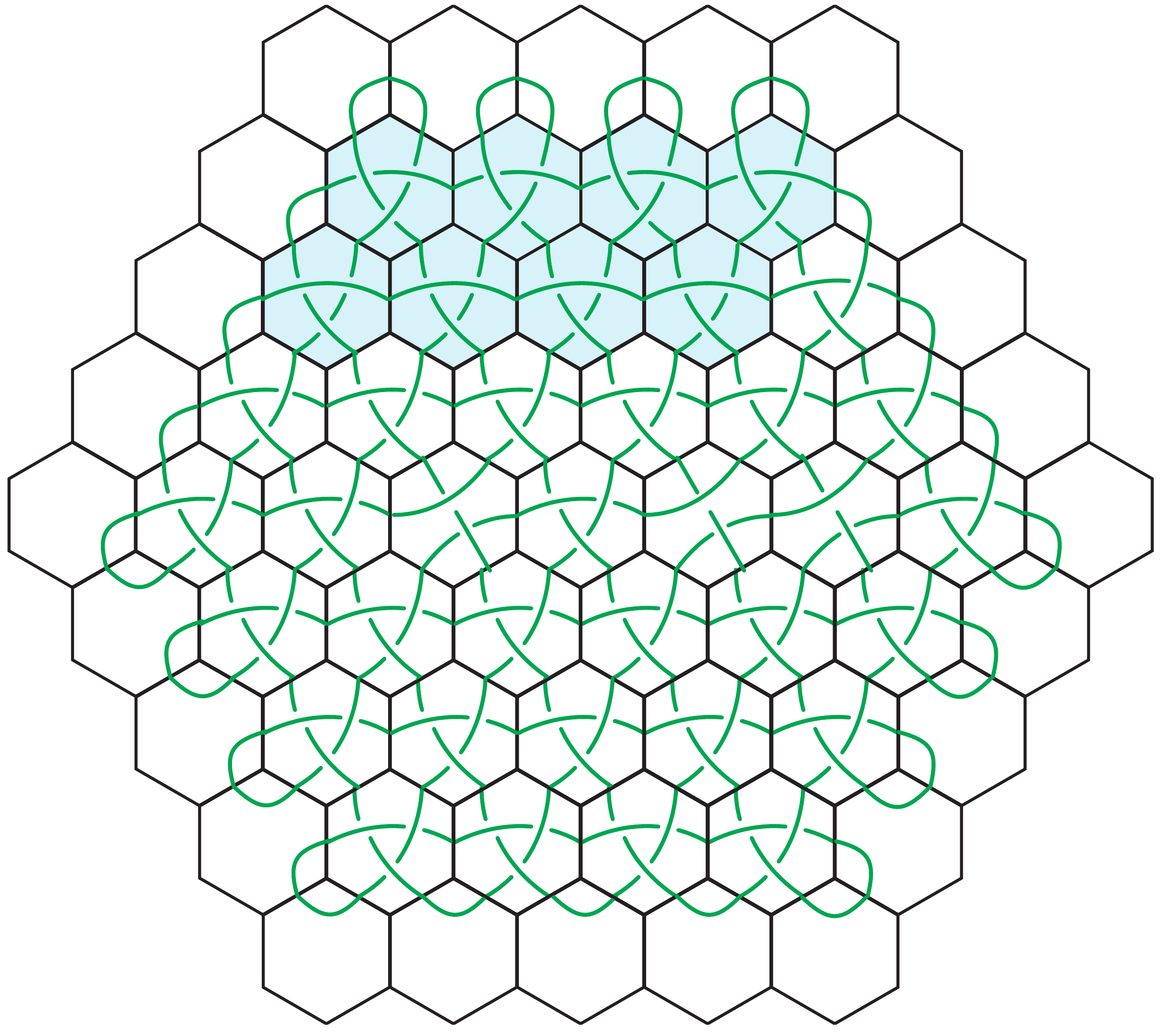}}
	\put(3.75,-0.29)  {\includegraphics[width=.42\textwidth]{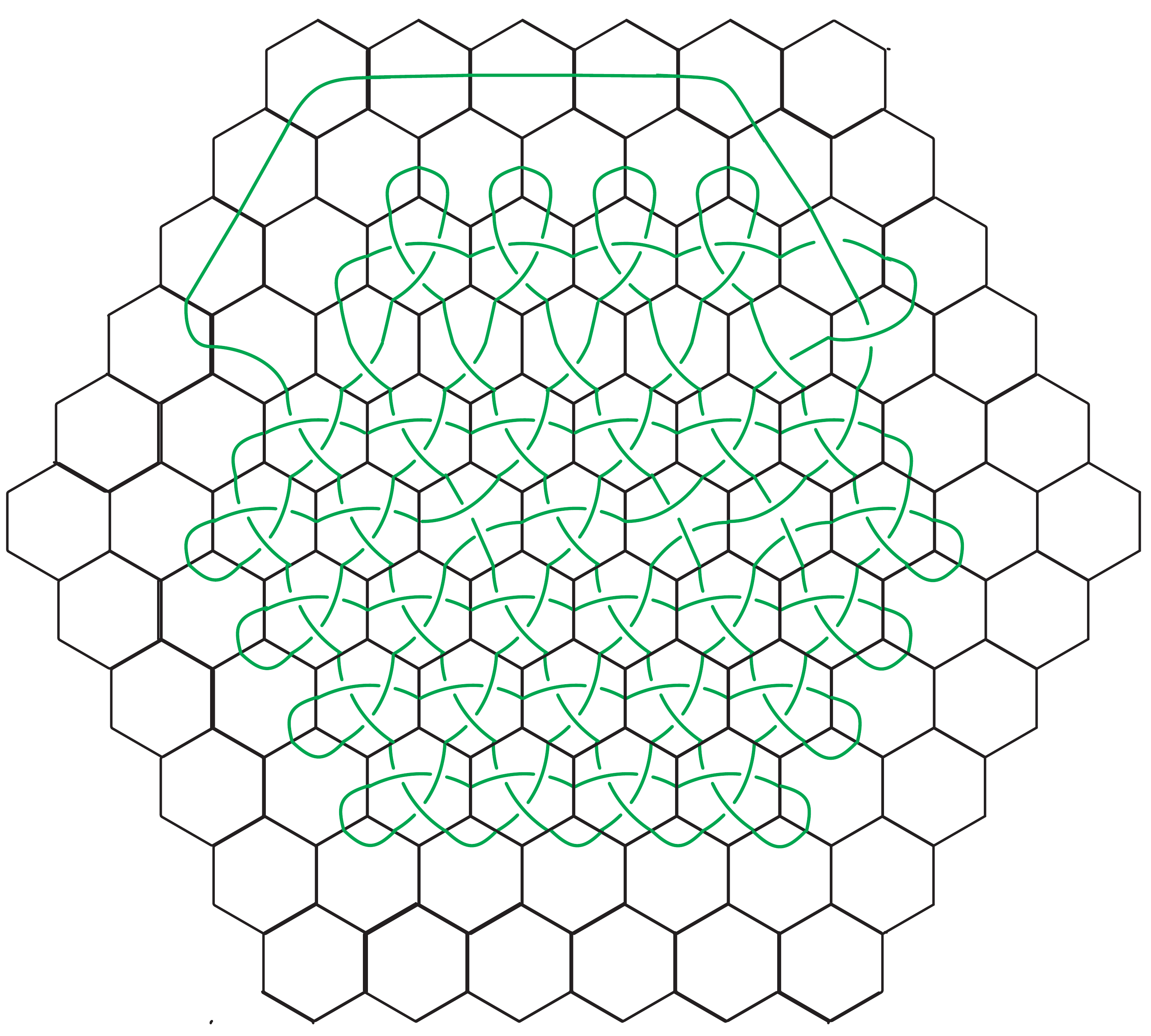}}
	\put(0,3.5){\includegraphics[width=.37\textwidth]{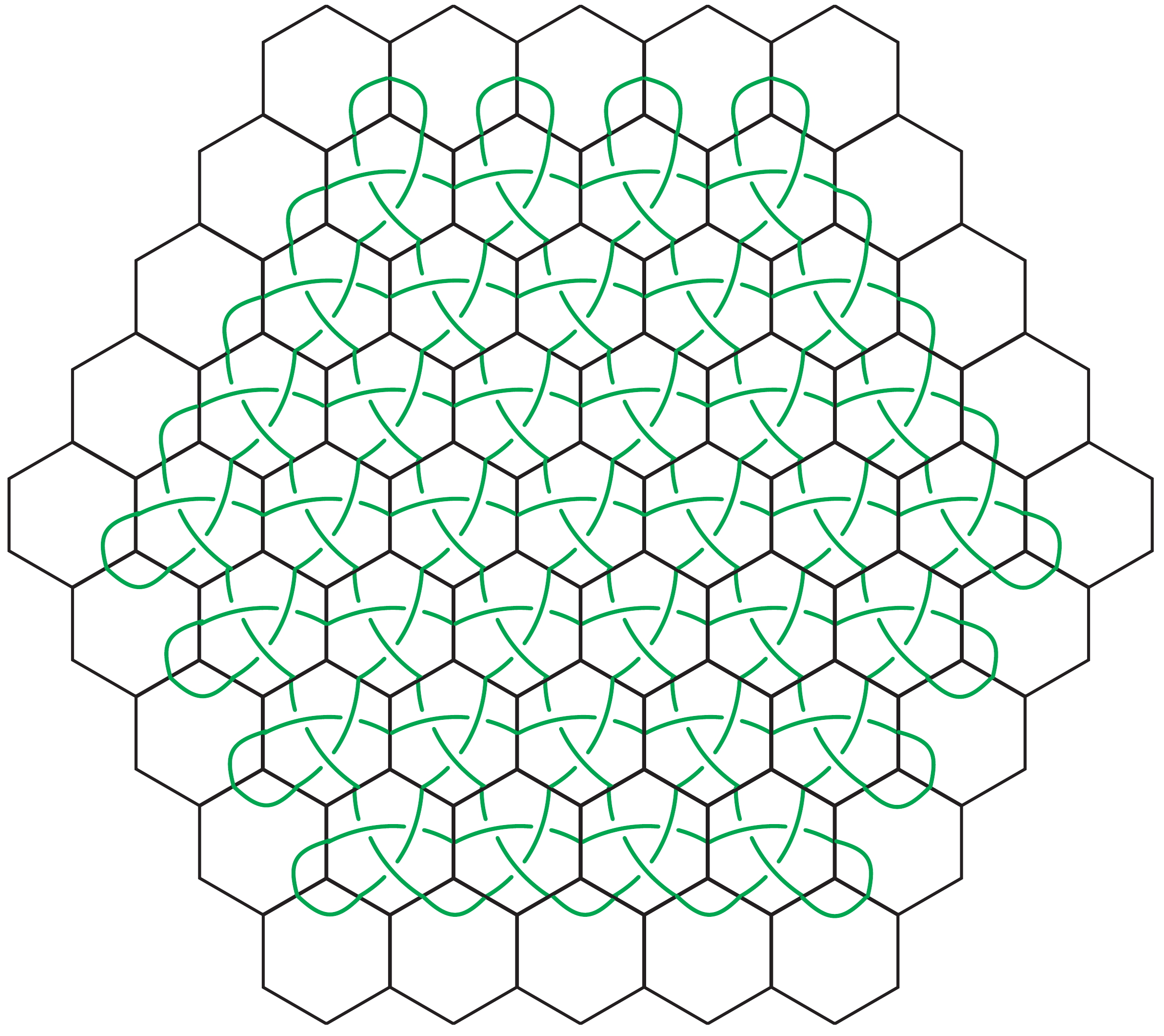}}
	\put(3.75,3.6)  {\includegraphics[width=.37\textwidth]{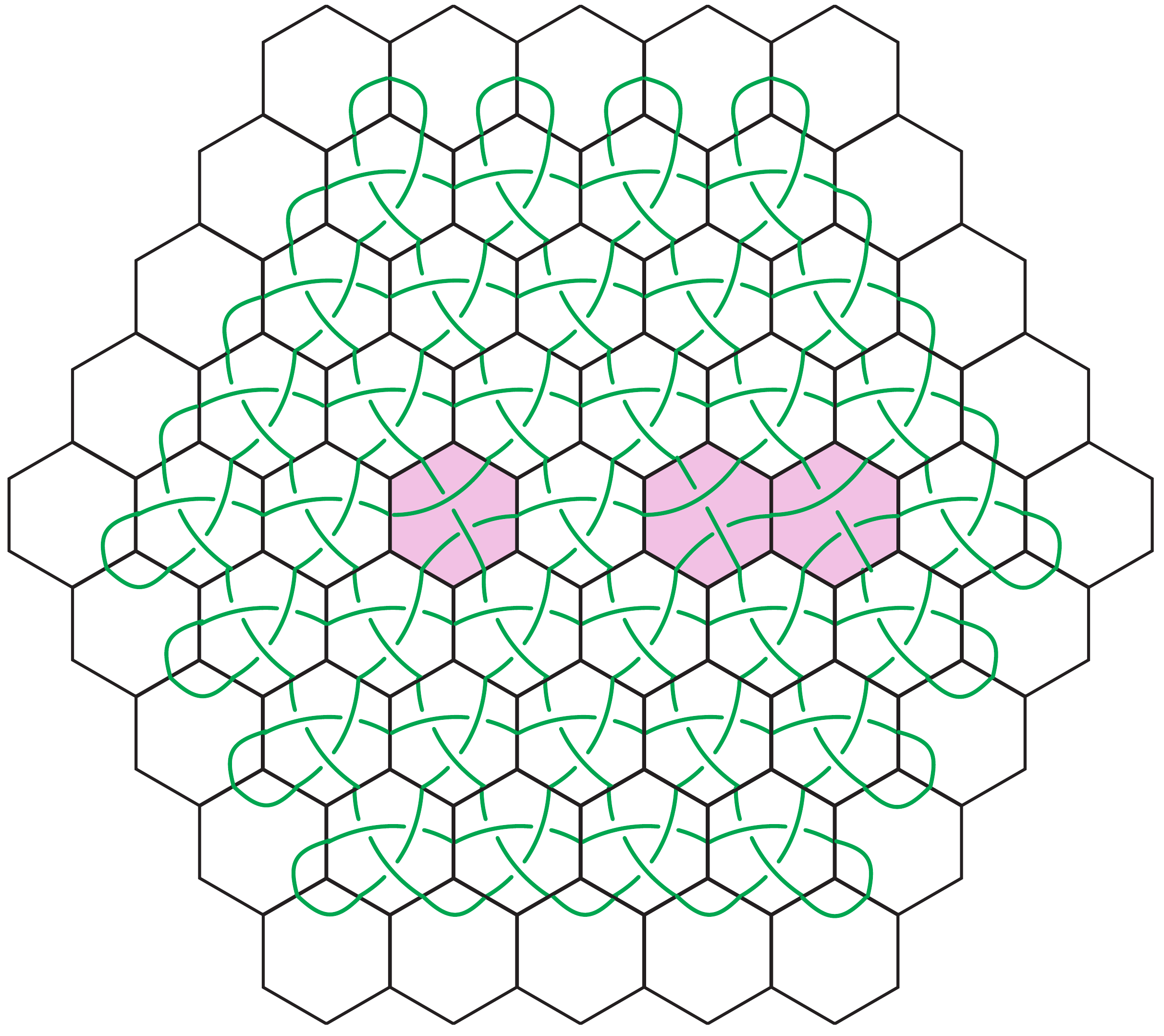}}
	\put(0.25,6.25){$L _5$}
	\put(4,6.25){$A_5$}
	\put(0.25,2.6){$K_5$}
	\put(4,2.6){$K_5$}
	\end{picture}
	\caption{Here $L_5$ (top left) is converted into an alternating knot $A_5$ (top right)
		by smoothing crossings, then the second highest horizontal strand is changed so that it can be lifted up (bottom left), and the crossings above it are changed to form $K_5$  (bottom left) so that after an isotopy of the raised arc we get a reduced alternating version of the knot. $K_5$ in that alternating form is pictured on a 6-mosaic in the bottom right.  The crossings changed to form $A_5$ from $L_5$ are shaded in pink the ones creating $K_5$ from $A_5$ are shaded in blue.  The same general strategy is used for $K_r$ for all $r \geq 4$}
	\label{fig:l5k5}
    \label{fig:k5}
\end{figure}

We define $K_3$ to be the knot shown in Figure~\ref{fig:k3} (as we will see it cannot be constructed quite in the same way as the other knots in the family). 
The process for creating $K_5$ from $L_5$ is depicted in Figure~\ref{fig:l5k5}.  The same strategy is used for $K_r$ for all $r \geq 4$.  Here the algorithm is detailed in general.
We can take the link $L_r, r \geq 4$ and form an alternating knot $A_r$ from it by removing $r-2$ of the existing central tiles (all of which are equivalent to Tile 27 in Figure~\ref{fig:tiles}) and replacing each of them with Tile 21 from Figure~\ref{fig:tiles}, a 2 crossing tile rotated to make sure $A_r$ is alternating.  We choose to replace central tiles in the lower half of the mosaic board.
The replacement is equivalent to smoothing one of the three crossings on each of these tiles.   Any time we smooth a crossing between two different components we reduce the number of components of the link by one (if both arcs of the crossing were from a single component smoothing might increase the number of components or leave the number unchanged so we avoid doing this).  By smoothing $r-2$ well chosen crossings we change $L_r$ from an $r-1$ component link into a knot.  
The transition from $L_5$ to $A_5$ is depicted in the top right mosaic in Figure~\ref{fig:l5k5} done on the tiles shaded pink.   It is not important which tiles we swap when forming $A_r$ as long as they are central tiles and thus not too close to the outside of the mosaic board, they are not in the top few rows, and they reduce the number of components when smoothed. There are plenty of choices for crossings to smooth to form $A_r$, but we note that  the collection of knots one could choose for $A_r$ defined here is a proper subset of the collection options for $A_r$ in \cite{hll} because in that paper there was no reason to disallow smoothing of crossings that were in the upper few rows of tiles of the mosaic board.

Next we form $K_r$. To do this we take $A_r$ and look at the second highest horizontal arc in the mosaic.  
It is fully contained in the third highest row of tiles in the mosaic.
  Fix the crossing all the way on the right, but change every other under-crossing of that strand into an over-crossing as seen in the bottom left mosaic board in Figure~\ref{fig:l5k5}.  Finally, we reverse all the crossings of the mosaic that lie above that horizontal arc to yield $K_r$ (these two sets of changes are shown shaded in blue in the bottom left in Figure~\ref{fig:l5k5}).

   Note that $K_r$ is an alternating knot, but it is not initially in an alternating diagram.  
$K_3$ is depicted in Figure~\ref{fig:k3} and $K_5$ in  Figure~\ref{fig:l5k5}.  
  While we prove in Theorem~\ref{thm:fit} that for all $r \geq 3$, $K_r$ cannot be embedded in a reduced alternating form on an $r$-mosaic, looking at this figure it is not hard to see that like $K_5$ each $K_r$ does fit on an $r+1$-mosaic while reduced and alternating.  
In Figure~\ref{fig:k5} we show $K_5$
both in its original diagram on a 5-mosaic and after lifting the strand to make it alternating on a 6-mosaic.
$K_3$ is also depicted in an alternating diagram in Figure~\ref{fig:k3dual}.  In general for $r \geq 4$, $K_r$ is always constructed so that the second highest horizontal strand can be moved up to yield an alternating diagram while decreasing the number of crossings by $2r-2$.  The knot diagram (ignoring the mosaic on which it is embedded) found by lifting the strand to find an initial alternating embedding of $K_r$ (shown for $K_3$ and $K_5$ in Figures~\ref{fig:k3dual} and \ref{fig:k5}  respectively) is called the \emph{standard alternating embedding of $K_r$}.

The number of choices for valid crossings to smooth increases as $r$ increases and as a result $K_3$ is the only knot in the family that cannot be formed using this construction. The obstacle for $K_3$ is that in $L_3$ there is only one central tile and only a single component enters it so one cannot smooth a crossing in a central tile of $L_3$ to form a knot, but for every other $K_r$, 
all the link components cross each other in the central tiles on the lower half of the board and
the smoothings can be done strictly in those tiles.  Because the number of central tiles relates to $r^2$, but the necessary number of smoothings is less than $r$, we have more and more good choices for crossings to smooth to form $K_r$ as $r$ goes up.  Thus as $r$ gets large there are many smoothing choices that would lead to a valid choice for $A_r$ and in turn $K_r$.

\section{The Complement of a Hexagonal Mosaic and the algorithm to construct Mosaic $L'$ from $L$.}
\label{sec:lp}

In \cite{hk}  the complement for a rectangular mosaic was introduced.  In \cite{hll} the idea was extended to hexagonal mosaics and we use it again here.
 The complement is only defined on the interior tiles.  Let $L$ be an embedding of a link on a hexagonal mosaic. Examine the interior tiles. On each of these tiles we pick arcs for the complement based on the arcs of $L$ already on that tile.  The choice for the arcs of the complement is dictated in Figure~\ref{fig:comp}.  
One simply replaces each interior tile from Figure~\ref{fig:tiles} with a corresponding tile from Figure~\ref{fig:comp} that contains the same arcs of the link, but also include arcs of the complement.  Looking at the tiles in Figure~\ref{fig:comp} one can see that the
 complement is defined so that on each interior tile it and the link together hit each the tile's 6 possible connection points (the midpoints of the hexagonal sides) exactly once, the arcs of the complement never cross each other, and if they cross an arc of $L$ the arc of the complement always goes under the arc of $L$.  Note that as the figure shows the complement is not always unique.  There are four tiles where there are a pair of choices one could make for the complement. These are the tiles that $L$ hits in at most one arc and thus the ones that contain at least two arcs of the complement (tiles 1 through 4 in Figure~\ref{fig:tiles}).  All the tiles with non-unique complement are depicted up to rotation in the top row of Figure~\ref{fig:comp}.  The tiles where the complement is uniquely determined (although perhaps empty) make up the other three rows.

\begin{figure}[tpb]
\centering
\includegraphics[width=.63\textwidth]{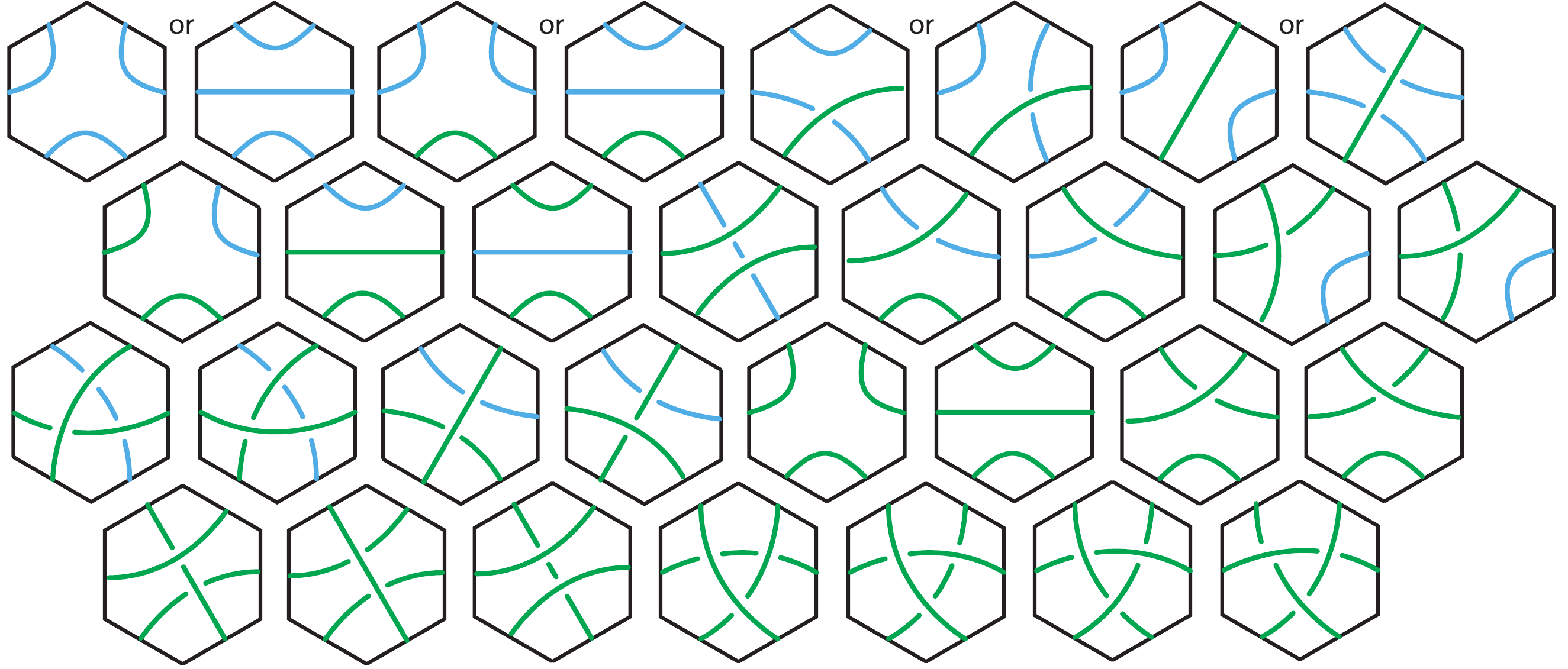}
\caption{The complement is in blue for green arcs of a mosaic.  When the link intersects the tile in a single arc there are usually two choices for the complement (pictured in the first three pairs).  Note that the complement is trivial on the 11 tiles that contain 3 arcs of the link since they already hit all 6 connection points for the hexagonal tile.}
\label{fig:comp}
\end{figure}

The complement is completely contained on the interior of the mosaic and while the original link, of course, is made up of a union of loops, the complement will consist of a (possibly empty) union of loops and arcs which are properly embedded on the interior of the mosaic.

\bigskip

\emph{Constructing Mosaic $L'$:}
Given any hexagonal $r$-mosaic  $L$  we define a mosaic $L'$ as follows.
\label{lemmalprime}
Let $L$ be a mosaic and let the set of arcs in the complement of $L$ be $\{a_1, a_2, \dots a_n \}$
and let the loops of the complement be $\{c_1, c_2, \dots c_j \}$.  Each arc $a_i$ divides the interior of $L$ into two pieces. Define the smaller piece to be the outside of $a_i$ and the larger piece to be the inside of $a_i$.  In our applications we will have a bound on the the length of the arcs of the complement which will dictate that every arc will have a well defined inside and outside.  Looking at $a_1$ and $a_2$ in the mosaic $L$ on the top left of Figure~\ref{hat} $a_1$ is outside of $a_2$.  We note that in the study of incompressible surfaces in the absence of loop $c_1$ both $a_1$ and $a_2$ would be considered outermost on the interior tiles since they each would have one side on which they were disjoint from the other arcs of the complement, but in our context $a_2$ is not considered outermost since each arc is defined to have a unique outside.

To form $L'$ take the interior arcs of $L$ (its intersection with the interior tiles), add the entire set of loops of the complement as well as arc $a_1$ to form the interior arcs of a new  link.  This defines the interior tiles of a new mosaic $L^1$ with at least as many crossings as $L$ (more if $a_1$ or any of the loops of the complement crossed $L$).  Now that the interior of $L^1$ is specified there are, as always in the standard setting, two ways in which it can be connected to itself in the boundary tiles.  One way will leave the boundary tiles outside of $a_1$ fixed as they were in $L$ and shift the connections in the boundary tiles on the inside of $a_1$.  The other way will shift the boundary tiles outside of $a_1$, but will fix the tiles on the inside of $a_1$.  The mosaic formed by the latter shift is called $L^1$.  See the top two mosaics in Figure~\ref{lp} for an example of the formation of $L^1$ from $L$.  In the figure the link consists of green and gray arcs and the complement is always in blue (gray is simply used to indicate the portions of the link which are changed relative to the previous picture).  Note that $L^1$ now has complement $\{a_2, a_3, \dots a_n \}$.  We have removed one arc from the complement and all of the loops.   Now form $L^2$ by adding $a_2$ to the interior arcs of $L^1$ in the same manner as we previously added $a_1$ to $L$.  $L^2$
has complement $\{a_3, a_4, \dots a_n \}$.  Repeat until we have no more arcs in the complement.  Call the final result $L'$.  
 See the first three mosaics in Figure~\ref{lp} for an example of forming $L'$ from a mosaic $L$. 
Since there are no arcs of the complement left after our construction, $L'$  must be saturated.

\begin{figure}[!ht]
   \centering
   \setlength{\unitlength}{0.1\textwidth}
   \begin{picture}(7,7)
    \put(0,.05){\includegraphics[width=.75\textwidth]{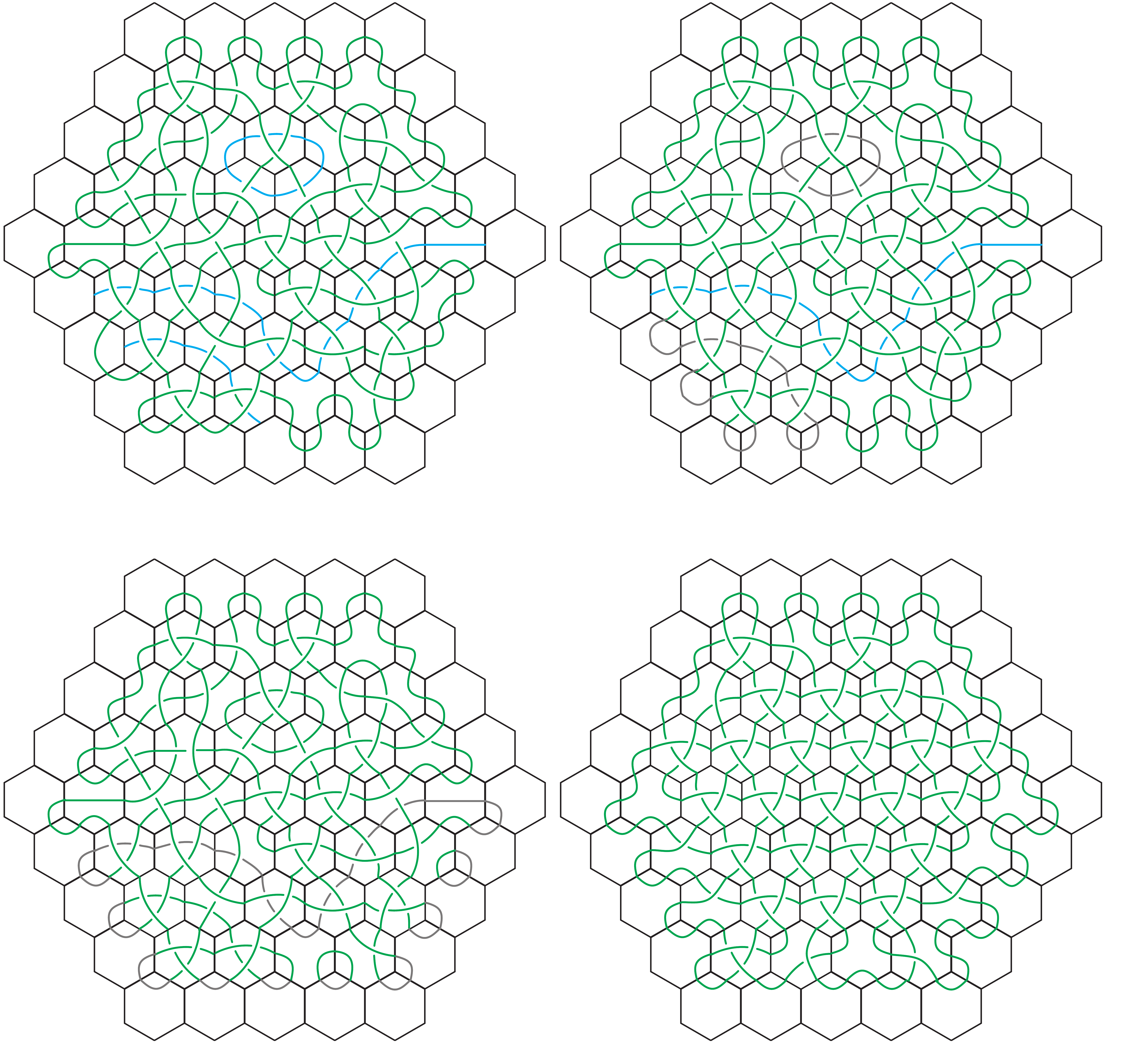}}
    
     \put(0.3,6.4){$L$}
     \put(4.0,6.4){$L^1$}
     \put(-.20,2.8){$L^2=L'$}
     \put(4.0,2.8){$\widetilde{L}$}
     \put(4.1,5.1){{\scriptsize $a_2$}}
     \put(5.4,4.15){{\scriptsize $a_1$}}
     \put(1.7,4.15){{\scriptsize $a_1$}} 
     \put(0.4,5.1){{\scriptsize $a_2$}}
     \put(1.37,6.1){{\scriptsize $c_1$}}
     \put(3.05,5.33){{\scriptsize $B$}}
     \put(0.75,5.2){{\scriptsize $C$}}
     \put(1.05,4.3){{\scriptsize $F$}}
     \put(2.6,5.45){{\scriptsize $A$}}
     \put(1.57,6.5){{\scriptsize $J$}}
     \put(1.95,6.55){{\scriptsize $I$}}
     \put(2.3,6.5){{\scriptsize $D$}}
     \put(2.552,6.15){{\scriptsize $G$}}
     \put(2.8,5.7){{\scriptsize $H$}}
     \put(2.75,5.05){{\scriptsize $M$}}
     \put(2.64,4.78){{\scriptsize $K$}}
     \put(2.33,4.35){{\scriptsize $E$}}   
     \put(3.0,1.66){{\scriptsize $B'$}}
\put(0.73,1.5){{\scriptsize $C'$}}
\put(1.05,0.65){{\scriptsize $F'$}}
\put(2.6,1.8){{\scriptsize $A'$}}
\put(1.54,2.85){{\scriptsize $J'$}}
\put(1.95,2.88){{\scriptsize $I'$}}
\put(2.3,2.85){{\scriptsize $D'$}}
\put(2.552,2.5){{\scriptsize $G'$}}  
\put(2.7,2.05){{\scriptsize $H'$}}
\put(2.75,1.4){{\scriptsize $M'$}}
\put(2.65,1.1){{\scriptsize $K'$}}
\put(2.31,0.67){{\scriptsize $E'$}}   
     \put(6.7,1.66){{\scriptsize $\widetilde{B}$}}
\put(4.43,1.34){{\scriptsize $\widetilde{C}$}}
\put(4.7,0.65){{\scriptsize $\widetilde{F}$}}
\put(6.18,1.84){{\scriptsize $\widetilde{A}$}}
\put(5.24,2.83){{\scriptsize $\widetilde{J}$}}
\put(5.65,2.85){{\scriptsize $\widetilde{I}$}}
\put(6.0,2.85){{\scriptsize $\widetilde{D}$}}
\put(6.22,2.5){{\scriptsize $\widetilde{G}$}}  
\put(6.4,2.18){{\scriptsize $\widetilde{H}$}}
\put(6.35,1.4){{\scriptsize $\widetilde{M}$}}
\put(6.2,.99){{\scriptsize $\widetilde{K}$}}
\put(6.01,0.8){{\scriptsize $\widetilde{E}$}} 
     \put(1.7,0.45){{\scriptsize $a_1$}} 
\put(0.35,1.45){{\scriptsize $a_2$}}

   \end{picture}
\caption{Given a link $L$
the arcs of the complement are drawn in blue and labeled $a_1$ and $a_2$ and the loop of the complement is labeled $c_1$.  
 We move $c_1$ from the complement to the link and add in arc $a_1$ to form $L^1$, then we add $a_2$ to $L^1$ to form $L^2=L'$.  Finally we show $\widetilde{L}$ which will be constructed out of $L'$ in the proof of Theorem~\ref{thm:none}.  The complement is always drawn in blue, but 
as arcs are changed or moved from the complement to the link we initially draw them in gray instead of green to help track the alterations from the previous figure. 
 We have also labeled many individual tiles that we refer to later.}
\label{lp}
\label{hat}
\end{figure}

\section{The Dual Graph for a link diagram and the exterior degree $\Delta$}

In this section we introduce a tool that we will use both to systematically find all possible flypes for a diagram of a link (the link need not be a mosaic link) and to analyze properties of the complementary regions to the diagram.   The Tait flyping conjecture states that all  reduced and alternating diagrams of a prime knot in $S^3$ can be obtained from each other by a sequence of flypes \cite{MT}.  
In a complicated diagram of a knot, however, it is easy to overlook a potential flype.  The Perko Pair is perhaps the most famous example of an overlooked flype.  The two diagrams appeared listed as distinct knots in knot tables including Conway's and Rolfsen's \cite{Rolfsen} and it was not until 1973 that Perko discovered that they were the same knot shown in two diagrams that differed by a flype \cite{Perko}.  We see a flype overlooked in an even more recent example in the paper that inspired this one \cite{L} where an overlooked flype requires an added argument to complete their proof.  In that case it is a minor error and the resulting gap in the proof of the theorem would be easy to fix, but it further points to the need for a systematic way to detect all flypes in a link diagram.  The dual graph developed in this section solves this issue.

Given  a diagram of a link, knot theorists usually record crossing information
so one can recreate the link, but at times we refer to the link projection and do not worry about over and 
under crossings.  The projection is a 4 valent graph on a 2 sphere (or in $\R^2$).  The projection then has a natural planar \emph{dual graph} $D$ which has a vertex in each region of the projection's complement and an edge connecting any pair of vertices that are separated by an edge of the projection graph.  While the complement of a hexagonal mosaic is a relatively new tool the dual graph is, of course, an old and well known concept in graph theory.  See Figure~\ref{fig:k3dual} for the dual graph of the standard alternating diagram of $K_3$.

The dual graph $D$ need not be 4-valent.  The degree of a vertex is equal to the number of edges of the projection which surrounds the region to which that vertex corresponds. Note that an inverse exists taking you from the dual graph uniquely to the projection from which it came, but the crossing information from the original diagram is lost (this is fine for our purposes).

Given $L$, a diagram of a link in the plane, and its dual graph $D$, we define the vertex corresponding to the infinite region to be the \emph{exterior vertex}.  It will usually be labeled  $v_1$.  
We define the degree of the exterior vertex to be the \emph{exterior degree} which we denote $\Delta(L)$.  Note that the exterior degree will be exactly the same as the number of sides of the outer polygon used in Ludwig, Evans and Paat's proof in \cite{L}. The dual graph in Figure~\ref{fig:k3dual} has exterior degree 7.  The \emph{maximal degree} of the  graph is defined to be the degree of the graph's highest degree vertex.  The exterior degree is, of course, bounded from above by the maximal degree.  We will show for $K_r$ that the highest exterior degree we can achieve in any reduced, alternating embedding is achieved in the standard alternating diagram (and that the maximal degree equals the exterior degree in those embeddings).

\begin{figure}[tpb]
\centering
\includegraphics[width=.75\textwidth]{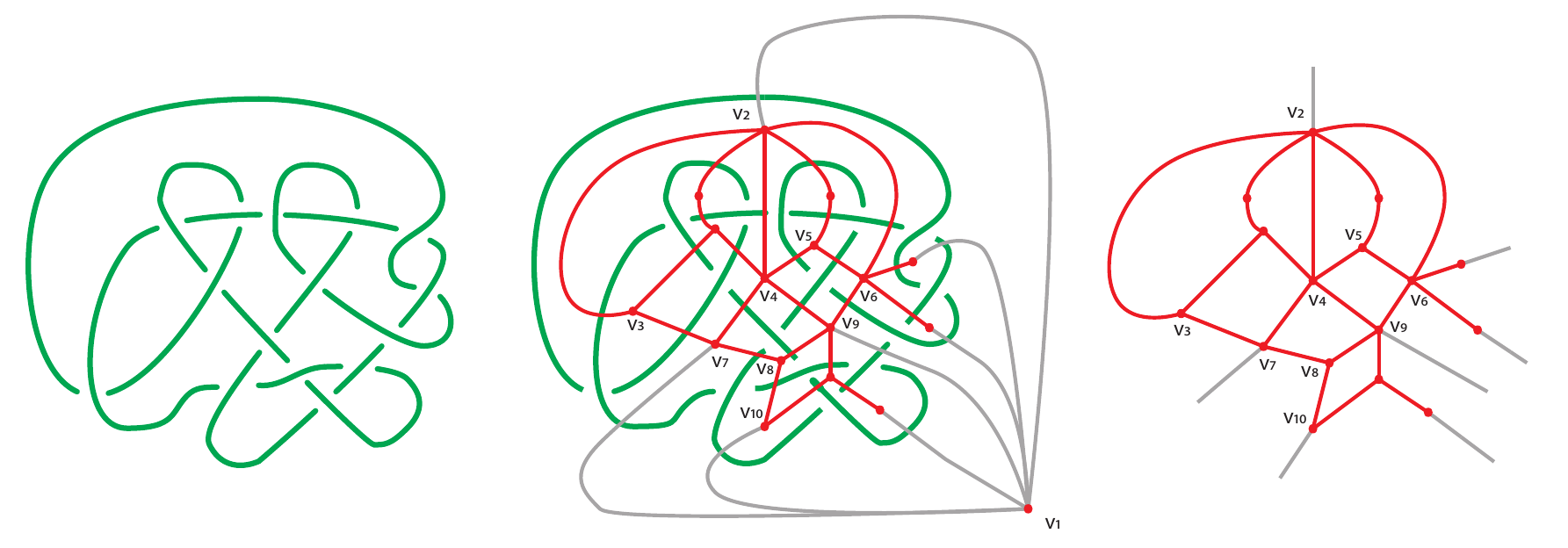}
\caption{The standard alternating embedding of $K_3$ is pictured left.  The dual graph is pictured center.  Because the degree of the exterior vertex ($v_1$ here) gets quite large for $K_r$ as $r$ increases, we usually draw the start of the edges running to the exterior vertex in gray, but omit the actual vertex to simplify the picture as seen on the right for $K_3$.}
\label{fig:k3dual}
\end{figure}

Tait's flyping conjecture is naturally set on $S^2$, not in $\R^2$, but mosaics are usually thought of as living in the plane so at times it will be convenient to think in one context and at times the other.  In $S^2$, of course, there is no exterior vertex so for diagrams on $S^2$ we simply choose a vertex of maximal degree to be the exterior vertex.

In a flype of a reduced, alternating diagram of a prime knot there are several circles that intersect the link diagram in four points that correspond to 4-cycles in the dual graph as seen in Figure~\ref{fig:flype}.  One circle bounds a disk whose interior contains just the crossing that is to be moved to the other side of the tangle by the flype (corresponding to the cycle $v_1$, $v_2$, $v_5$, $v_4$ in the graph on the left in Figure~\ref{fig:flype}), one bounds a disk that contains the tangle that is to be flipped over ($v_2$, $v_3$, $v_4$, $v_5$ in the same graph), and one bounds a disk that contains both the crossing and the tangle ($v_1$, $v_2$, $v_3$, $v_4$ again in the same graph).  The union of the three corresponding 4-cycles always contains a total of 5 vertices and they are called \emph{nested 4-cycles}. The cycle bounding the disk containing both the crossing and the tangle is called the
  \emph{outer 4-cycle}  and the other two are called \emph{inner 4-cycles}.

Note that one might worry that a priori it is possible that either $v_1$ and $v_3$ or $v_2$ and $v_4$ actually lie in the same region of the projection complement and thus should not represent distinct vertices.  For a reduced link on $S^2$, however, the only way this can happen is if the link has no crossings except the ones involved in the flype as in Figure~\ref{fig:isot}.  In that case the new
diagram of the knot is equivalent to flipping over the entire diagram of the knot.  While this technically may be a new diagram  it is not different in any interesting way from the original so we call this a trivial flype and only focus on all other types of flypes.

\begin{figure}[tpb]
\centering
\includegraphics[width=.51\textwidth]{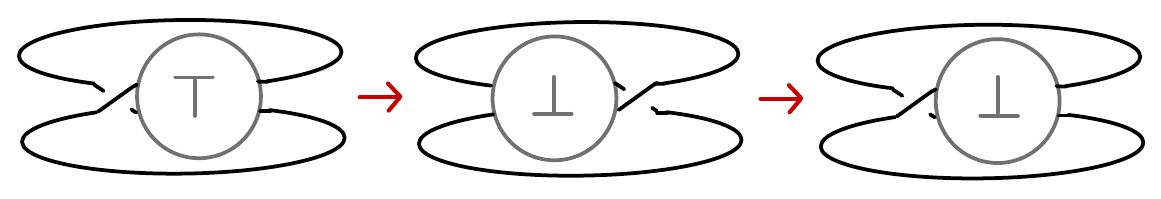}
\caption{Here a flype takes us from left to center, but an isotopy moves the crossing back to the original side of the tangle yielding the picture on the right.  This is equivalent to flipping over the entire link diagram and is considered a trivial flype.}
\label{fig:isot}
\end{figure}

 In a non-trivial flype one of the two inner 4-cycles will bound a disk in $S^2$ whose interior is disjoint from the rest of the dual graph
($v_1, v_2, v_5, v_4$ for the dual graph on the left in Figure~\ref{fig:flype}).  
We will call any 4-cycle such as $v_1, v_2, v_5, v_4$ which bounds a disk in $S^2$ (or in $\R^2$ if the diagram is in the plane) whose interior is disjoint from the rest of the dual graph an \emph{empty 4-cycle}. We see that for a reduced alternating diagram every non-trivial flype will require one empty (inner) 4-cycle and two more 4-cycles, one non-empty inner 4-cycle and one non-empty outer 4-cycle.

\begin{figure}[tpb]
\centering
\includegraphics[scale=.4]{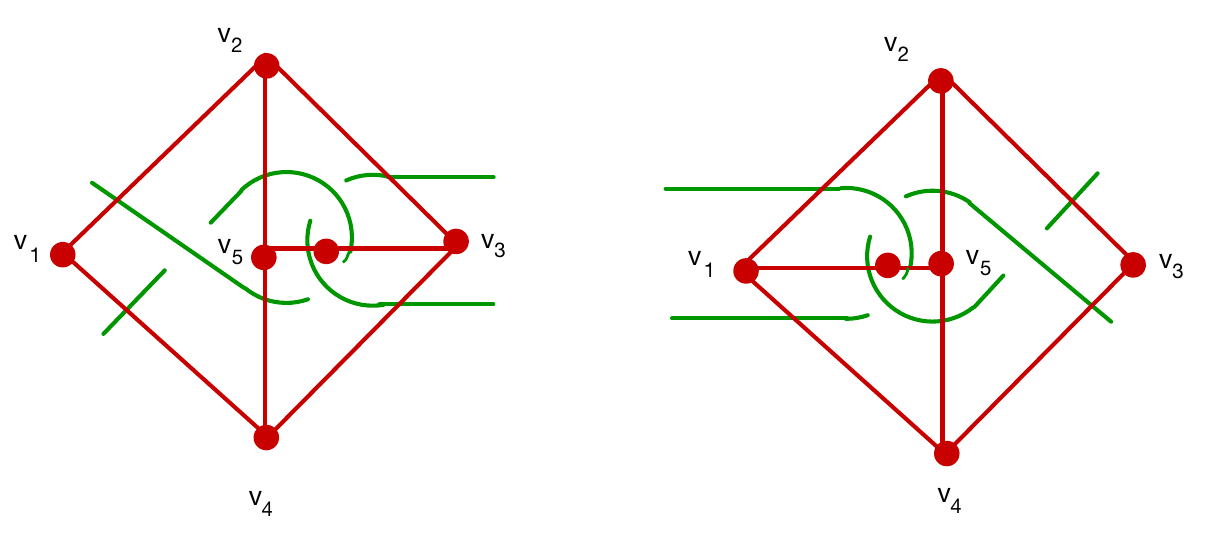}
\caption{Because flypes are simple moves the influence on the dual graph is always predictable. For every possible flype for $K_r, r \geq 4$ a subgraph of the dual graph like the one on the left is always replaced by a sugbraph like the one on the right  - the subgraphs are isomorphic to each other, but the degrees of some vertices do change. $K_3$ is different from the other $K_r$'s  with one slightly more complicated flype available.}
\label{fig:flype}
\end{figure}

If before the flype the empty inner cycle contains vertices $v_1, v_2, v_5, v_4$ and the non-empty inner cycle contains $v_2, v_3, v_4, v_5$  then after the flype they are reversed and the empty inner cycle contains vertices $v_2, v_3, v_4, v_5$ and the non-empty inner cycle contains $v_1, v_2, v_5, v_4$. Note that the flype always causes the degree of the vertex corresponding to $v_1$ 
 in the figure to go up and the degree of the degree of the vertex corresponding to $v_3$ to go down.  One can also easily show that all five vertices in the 4-cycles must have degree at least 3 so when searching for 4-cycles corresponding to flypes we can ignore vertices of degree 2 making the process even faster.

Thus for reduced, alternating diagrams of links we can systematically find all flypes by searching for nested 4-cycles in the dual graph  (the test might yield a set of nested 4-cycles that do not correspond to a flype, but the crucial thing is that every flype will yield a set of nested 4-cycles).  Since every set of nested 4-cycles contains non-empty 4-cycles which are in turn rare in dual graphs we search for non-empty 4-cycles and this usually gives a short list that is easy to check for corresponding flypes. 
We search the dual graph $D$ for all non-empty 4 cycles containing $v_1$ then search  $D - v_1$ for all such 4-cycles containing $v_2$ and so on. We apply this search in the proof of Claim~\ref{claim:highest} finding that we have no non-empty 4-cycles in $D - (v_1 \cup v_2)$ showing there are no nested 4-cycles away from $v_1$ and $v_2$ and thus every flype corresponds to a cycle containing at least one of those two vertices.

\section{Proof that $K_r$ cannot be embedded on a hexagonal $r$-mosaic while achieving crossing number.}
In this section we extend the main result of \cite{L} from rectangular mosaics to standard hexagonal mosaics.  First we establish some useful results. Lemma~\ref{lemma:comp} was proven in \cite{hll} and is also a direct result of the main theorem in \cite{jmm}.   Claims~\ref{lr} and \ref{claimlr}
are easy to prove (see \cite{hll}).

\label{sec:lud}

\begin{lemma}
Saturated links on a standard hexagonal $r$-mosaic  (including $L_r$ as well as the other saturated options)
 have $r-1$ components for $r \geq 2$.  
 \label{lemma:comp}
\end{lemma}

Note that  in $L_r$ for $r \geq 4$ because of the shape of the link components described above  every given pair of components  crosses each other on the central tiles in the lower half of the mosaic, which makes it easy to find crossings to smooth when forming $A_r$.

\begin{claim}
Given a hexagonal $r$-mosaic, $r \geq 2$, there are $3r^2-9r+7 $ interior tiles and $L_r$ 
has $9r^2-27r+21$ crossings in the standard embedding.  
\label{lr}
\end{claim}

\begin{claim}
Let $L_r$ have crossing number $s_r$. We observe $s_r = 9r^2-27r+21$ for $r \geq 3$.
\label{claimlr}
\end{claim}

Claim~\ref{lr} implies that the number of central tiles in the bottom half of the mosaic grows more quickly than $r^2$.  Since every pair of components of of a given $L_r, r \geq 4$ cross each other in this collection of tiles we can use only these tiles to smooth the $r-2$ crossings necessary to get the knot $A_r$ as desired for $r\geq 4$.

\begin{claim} $A_r$ has crossing number $9r^2-28r+23$ and $K_r$ has crossing number $9r^2-30r+25$ for $r \geq 4$.   $A_3$ has crossing number 19 and $K_3$ crossing number 15.
\label{claim:krarcn}
\end{claim}

\begin{proof}
For $r \geq 4$ the proof follows from the constructions in Section~\ref{sec:fam} which gave $A_r$ in a reduced alternating form and from which it was easy to find a reduced alternating diagram of $K_r$.  $K_3$ and $A_3$ can be checked explicitly from the diagrams in Figure~\ref{fig:k3}.
\end{proof}

\begin{claim} 
The dual graph for $L_r,  r \in Z, r \geq 2$ (in the standard diagram) has exterior degree  $9r-15$.
\label{claim:lrextdeg}
\end{claim}

\begin{proof}

The exterior degree of the dual graph of $L_r$ is, of course, equal to the number of edges in the exterior face of the 4-valent graph projection of $L_r$ from which the dual graph is derived.  Below we will count those edges.

 The tiles containing the outside edges for $L_r$ alternate between 2 types of patterns as seen in $L_3$, $L_4$, and $L_5$ pictured in Figures~\ref{fig:l3}, \ref{fig:l4}, and \ref{fig:l5k5} respectively.  These correspond to type I edges and type II edges.  The pattern continues for any radius and just gets bigger adding one tile on each of the six sides each time $r$ goes up by one.  There is a bijection between vertices and edges on the exterior face of the link projection and so counting the vertices suffices to find the exterior degree.  For $r=3$ the mosaic has three sides containing four vertices each and three sides containing two vertices.   Six vertices  in the corners are double counted yielding a total of 12 distinct vertices (and therefore edges).  For $r=4$ we have three sides with six vertices (the sides containing type II tiles) and three sides with three vertices (the sides containing type I tiles).   Six vertices are double counted yielding a total of 21.  
In general as $r$ increases by one the first set of sides gains one more type I tile yielding two vertices and the second gains one more type II tile increasing the number of vertices on that side by one, and they always share six vertices in common, so the total number of vertices in the n-gon goes up by nine.  See Figure~\ref{fig:l5dual} for the dual graph of $K_5$ with exterior degree 30.

\end{proof}

\begin{claim}
In the standard alternating diagram of $K_r$, the dual graph has exterior degree  $7r-13$ for $r \geq 4$ and 7 for $r=3$.
\end{claim}

\begin{proof}

Note that $r=3$ is a special case and it has exterior degree 7 as seen in Figure~\ref{fig:k3dual}.
In general for $r  \geq 4$ the embedding after isotopy has exterior degree $2r-2$ below the exterior degree of $L_r$ so we get degree $9r-15-(2r-2) = 7r-13$.  This is clear by noting that for $r \geq 4$ $A_r$ has the same exterior degree as $L_r$ since all smoothed crossings are on central tiles.  See Figure~\ref{fig:l5dual} for the dual graph of $A_5$.  $K_r$ has the same exterior degree as $A_r$ before the isotopy, but after the isotopy to the standard alternating diagram it drops by $2r-2$.  One can easily see that for $K_4$ the number drops by $6$ and that each time $r$ increases by one then two more edges of the link projection are removed from the exterior (corresponding to the added type II tile on the top side of the mosaic) than were removed in the previous case.

\end{proof}

At times we will want to compare the exterior degree and the number of crossings of two different mosaics, say $L$ and $L'$.  Recall we denoted the exterior degree of $L$ by $\Delta(L)$. Let the number of crossings of $L$ be called $Cr(L)$ (as opposed to $\Delta(L')$ and  $Cr(L')$ for $L'$).  Note that $Cr(L)$ is the number of crossings in the specific diagram $L$, not the crossing number of $L$ considered over all diagrams.  We will also want to examine how individual tiles contribute to $\Delta(L)$ and $Cr(L)$ to see what happens when we swap out tiles to form new mosaics.  Let $T$ be a given tile in $L$ and $T'$ be one in $L'$.  We let $Cr(T)$ be the number of crossings on tile $T$.  
We can also ask how much $T$ contributes to $\Delta(L)$. We know $\Delta(L)$ is equal to the number of edges in the boundary of the exterior face of the mosaic link projection.  
If the boundary of the exterior face of the link projection of $L$ has a vertex which corresponds to a nugatory crossing in $L$, then we say $L$ has a \emph{nugatory crossing on its exterior}. 
If $L$ does not have a nugatory crossing on its exterior then the boundary of the exterior face consists of a set of polygons (and a single polygon if the projection is connected as is the case for knots and non-split links).  There is a bijection between edges and vertices for polygons, so if $L$ has no nugatory crossings on its exterior then the exterior degree will also equal the number of vertices in the boundary of the exterior face of its projection.  
We let $\Delta(T)$ be the number of crossings on tile $T$ which correspond to those vertices and 
thus if $L$ has no nugatory crossings on its exterior 
then  $\Sigma( \Delta(T_i))$ over all tiles $T_i$ in $L$ must equal the exterior degree of $L$.

\begin{theorem}

Let $L$ be a standard hexagonal $r$-mosaic with no nugatory crossings on its exterior.  If  $r \geq 4$ and $Cr(L) \geq 9r^2-30r+25$ then $\Delta(L) > 7r-13$.   For $r=3$ if $Cr(L) \geq 15$ then $\Delta(L) > 7$.
\label{thm:none}
\end{theorem}

\begin{proof}

For $r \geq 4$ we will take a mosaic $L$   with no  nugatory crossings on its exterior such that $\Delta(L) \leq 7r-13$
and show $Cr(L) < 9r^2-30r+25$.  For $r=3$ we will take one with $\Delta(L) \leq 7$ and show $Cr(L)  < 15$.
In particular for $r \geq 4$ we will show that when compared to $L_r$, $L$
may be assumed to be missing at least $3r-3 = \frac{3}{2} *(2r-2)$ crossings on the penultimate tiles, meaning $Cr(L) \leq s_r-(3r-3)
= 9r^2-27r+21- (3r-3)= 9r^2-30r+24$.  The argument for $r=3$ is similar although the count is slightly different.
We will then show in Claim~\ref{claim:highest} that every possible minimal crossing embedding of $K_r$ on an $r$-mosaic would contradict this proving no such embedding exists.

Assume $L$ is a hexagonal $r$-mosaic for $r \geq 4$ with $\Delta(L) \leq 7r-13$ and $Cr(L) \geq 9r^2-30r+25$ or a 3-mosaic with $\Delta(L) \leq 7$ and $Cr(L) \geq 15$, and in either case with no nugatory crossings on its exterior and that $L$ was chosen to have the smallest number of arcs and loops in its complement over all such mosaics.  Case 1 will  show that we can assume there are no arcs or loops in the complement of $L$ and then Case 2 will show that if there are no arcs or loops in the complement we get a contradiction so no such mosaic exists.

\bigskip

{\bf Case 1: There is at least one arc or loop in the complement of $L$. }  

\bigskip

  We start with $L$ as described in the previous paragraph and form $L'$ using the construction in Section~\ref{sec:lp}.
  We know $Cr(L') \geq Cr(L)$, and $L'$ has no arcs or loops in its complement, but it is possible that $\Delta(L') > \Delta(L)$ or that $L'$ has a nugatory crossing on its exterior even though $L$ did not. 
We therefore alter $L'$ to form a mosaic $\widetilde{L}$ with no nugatory crossings on its exterior, $Cr(\widetilde{L}) \geq Cr(L)$
 and $\Delta(\widetilde{L}) \leq \Delta(L)$.  We will then  show that having trivial complement and low exterior degree also guarantees that  $Cr(\widetilde{L}) < Cr(L)$, the final contradiction proving $L$ does not exist.

 We form $\widetilde{L}$ from $L'$ by replacing all the central tiles that don't already have three crossings with tiles that do as well as prescribing a process to swap out the penultimate tiles of $L'$.
An example of forming $L'$ from $L$ and then  $\widetilde{L}$ from $L'$ is shown in Figure~\ref{hat}.  
 We only use three types of tiles for our substitutions  to form $\widetilde{L}$.
 When we refer to a inserting a zero crossing tile we will pick Tile 17 from  Figure~\ref{fig:tiles} (Tile 18 sometimes would be fine, but 17 always works so we stick with it every time).
 When we refer to a one crossing tile we use Tile 19, 
 but since the projections of tiles 19 and 20 are the same, just the crossing is reversed, either one would have been equally good. We will not use any two crossing tiles.
 Finally, when we insert a three crossing tile any of the three crossing tiles work fine, since again, their projections are all the same, but for consistency we use
Tile 27 from Figure~\ref{fig:tiles}.
Each time we insert a tile $\widetilde{T}$ into the penultimate corona, we rotate the tile to make sure no edge of the dual graph
can pass through $\widetilde{T}$ from the exterior vertex to the central tiles  without hitting one of the edges of the link projection in the tile.
 This will ensure that $\Delta(\widetilde{L})$ is exactly the sum of the values from the penultimate tiles.  
We also are careful that the substitution is rotated so that it never introduces nugatory crossings on the exterior.

 We now look at the specific rules for the substitutions.

\bigskip

\emph{Forming $\widetilde{L}$ from $L'$ - substitutions for the central tiles:}
The easiest substitution is on the central tiles (recall the central tiles include 
all the interior tiles except the penultimate tiles).  We replace all the central tiles with 3-crossing tiles. Since $L'$ was suitably connected and had no arcs or loops in its complement, it hit all of the connection points for the interior tiles and thus replacing these tiles of $L'$ with three crossing tiles leaves a link that is still suitably connected.
Although we pick the same tile every time we could equally well choose any 3-crossing tile and rotate it any way we want.
Note that this can never lower the number of crossings on a tile relative to $L$ or $L'$ and if $T$ intersects an arc of the complement of $L$ it has at most one crossing on it, so in that case we are gaining two crossings for each such tile relative to $L$ (see tiles $A$, $A'$, and $\widetilde{A}$ in Figure~\ref{hat}).

\bigskip

\emph{Forming $\widetilde{L}$ - substitutions on tiles that contained the endpoints of the arcs of complement of $L$:}  If there were no arcs in the complement of $L$ and only loops then we skip this step.  Otherwise take an arc of the complement $a_i$ and
let $T$ be one of the tiles in the penultimate corona of $L$ containing one of the endpoints of $a_i$.  Let $T'$ be the corresponding tile in $L'$.  
$Cr(T) \leq 1$ since the presence of $a_i$ ensures $T$ contains at most 2 arcs of $L$. If 
$Cr(T)=0$ then it trivially follows that $\Delta(T)=0$ as well.
In that case we replace $T' \subset L'$ with a 0 crossing tile $\widetilde{T}$ 
which then dictates that $Cr(\widetilde{T})= \Delta(\widetilde{T}) = 0$ matching the original contributions of $T$ (see tiles $B$, $B'$, and $\widetilde{B}$ in Figure~\ref{hat}).   Similarly if 
$Cr(T)=1$ then since $T$ only contains two arcs of $L$ and $L$ has no nugatory crossings on its exterior,  it follows that $\Delta(T)=1$.  In that case 
  we  can replace the corresponding tile $T'$ from $L'$  with a one crossing tile $\widetilde{T}$
  that    like $T$ has the property that  $Cr(\widetilde{T}) = \Delta(\widetilde{T}) = 1$ (see tiles $C$, $C'$, and $\widetilde{C}$ in Figure~\ref{hat}).

We now know that the substitutions for tiles containing the endpoints of the arcs of the complement change neither the number of crossings of the mosaic nor the exterior degree relative to the original mosaic $L$ so we now turn our focus to the substitutions for the penultimate tiles which did not contain endpoints of the complement of $L$ to make sure 
$Cr(\widetilde{L}) \geq Cr(L)$  and $\Delta(\widetilde{L}) \leq \Delta(L)$.
  
  \bigskip

\emph{Forming $\widetilde{L}$ - substitutions for corner tiles in the penultimate corona which did not contain any endpoints of the arcs of the complement of $L$:}
Our next group of tiles to fix are the corner tiles with the exception of those corner tiles which were already replaced in the previous step.
Let $T$ be a corner tile of the penultimate corona of $L$ and let $T'$ be the corresponding corner tile for $L'$.
We want to find a tile $\widetilde{T}$ to replace $T'$ with such that $ Cr(\widetilde{T}) \geq Cr(T)$ and  $ \Delta(\widetilde{T}) \leq \Delta(T)$.
 Each corner tile
 intersects three boundary tiles. If $\Delta(T) = 0$ then $Cr(T)=0$ and we replace $T'$ with a zero crossing tile
  so that $Cr(\widetilde{T}) = \Delta(\widetilde{T}) = 0$.  We are careful to rotate $\widetilde{T}$ so that it does not form a split link component with the boundary tiles (see tiles $D$, $D'$, and $\widetilde{D}$ in Figure~\ref{hat} for an example of this substitution).
Because $T$ is a corner tile and $L$  has no nugatory crossings on its exterior, if $\Delta(T) = 1$ then $Cr(T)=1$ in which case we replace $T'$ with a one crossing tile
positioned so that $Cr(\widetilde{T} ) = \Delta(\widetilde{T}) = 1$ and so it does not introduce a nugatory crossing in the exterior (see tiles $E$, $E'$, and $\widetilde{E}$ in Figure~\ref{hat}).  Finally, if $\Delta(T) = 2$ then we can choose $\widetilde{T}$ to be a three crossing tile with the property that $\Delta(\widetilde{T}) = \Delta(T) = 2$,
and since $Cr(\widetilde{T}) = 3$ we know $Cr(\widetilde{T}) \geq Cr(T)$
(see tiles $F$, $F'$, and $\widetilde{F}$ in Figure~\ref{hat}).

 Now we know that to make sure
$Cr(\widetilde{L}) \geq Cr(L)$  and $\Delta(\widetilde{L}) \leq \Delta(L)$
   we only need a substitution for the penultimate tiles that neither contained one of the endpoints of an arc of the complement of $L$ nor are corner tiles.  Since any remaining tile $T$ is disjoint from the endpoints of $a_i$, we know that $T \subset L$ and the corresponding $T' \subset L'$ are either type I or II (although they may or may not be the same type as each other).  Note that for 3-mosaics all the substitutions are now complete and one can skip directly to the conclusion of Case 1  because for $r=3$ every tile in the penultimate corona is a corner tile and the following situations cannot occur.

\bigskip 
 
\emph{Forming $\widetilde{L}$ - substitutions for type I tiles in the penultimate corona:}  Given corresponding tiles $T$ and $T'$,
 If $T' \subset L'$ is type I and $T$ does not contain an endpoint of one of the arcs of the complement and for   $T \subset L$, $\Delta(T) = 0$ then $Cr(T) \leq 1$.  In this case we replace $T'$ with a one crossing tile $\widetilde{T}$
rotated so that  $\Delta(\widetilde{T}) = 0$ and $Cr(\widetilde{T}) = 1$ (see tiles $G$, $G'$, and $\widetilde{G}$ in Figure~\ref{hat}).  Then we have $\Delta(\widetilde{T})\leq \Delta(T)$ and $Cr(\widetilde{T}) \geq Cr(T)$.  If, on the other hand, 
  $T' \subset L'$ is type I and for  $T \subset L$, $\Delta(T) \geq 1$ then we replace $T'$ with $\widetilde{T}$ a three crossing tile 
  rotated so that 
 $\Delta(\widetilde{T}) = 1$ (see tiles $H$, $H'$, and $\widetilde{H}$ in Figure~\ref{hat}).  Then since $Cr(\widetilde{T})=3$
 we know $\Delta(\widetilde{T})\leq \Delta(T)$ and $Cr(\widetilde{T}) \geq Cr(T)$.  Now we are left to worry about the case where $T'$ is type II.   
 
\bigskip

\emph{Forming $\widetilde{L}$ - substitutions for type II tiles in the penultimate corona:}  
 We are left with the case of corresponding tiles $T$ and $T'$  where $T'$ is type II and $T$ does not contain an endpoint of an $a_i$.  If $\Delta(T)  =  2$ then we replace $T'$ with a three crossing tile $\widetilde{T}$ 
 rotated so that 
 $\Delta(\widetilde{T}) = 2$ (see tiles $I$, $I'$, and $\widetilde{I}$ in Figure~\ref{hat}).  Then since $Cr(\widetilde{T})=3$
 we know $\Delta(\widetilde{T})\leq \Delta(T)$ and $Cr(\widetilde{T}) \geq Cr(T)$. 
 
  If $\Delta(T)  =1$ then we replace $T'$ with a one crossing tile $\widetilde{T}$ 
  rotated so that 
 $\Delta(\widetilde{T}) = 1$ (see tiles $J$ and $K$, $J'$ and $K'$, and $\widetilde{J}$ and $\widetilde{K}$ in Figure~\ref{hat}).  Then 
 we know $\Delta(\widetilde{T})\leq \Delta(T)$.  Since $T$ did not contain an endpoint of an $a_i$, it was either type I or type II.  If it was also type II, like $T'$,
 then $Cr(\widetilde{T}) \geq Cr(T)$.  If it was type I and an arc of some $a_i$ or $c_j$ in the complement passed through it, then it contained at most one crossing so again 
 $Cr(\widetilde{T}) \geq Cr(T)$.   A problematic situation that   is if 
 like tile $K$ in Figure~\ref{hat}, $T$ is disjoint from all the arcs and loops of the complement of $L$ in which case we may have lost up to 2 crossings with the substitution, for example, we see   $Cr(\widetilde{K})=Cr(K)-2$.   Lemma~\ref{addedcrossings} will say that earlier substitutions on the central tiles increased the number of crossings by at least enough to compensate for any lost crossings here.

The final possibility when $T'$ is type II  is that $\Delta(T)=0$.  In this case we substitute $\widetilde{T}$ a 0 crossing tile 
 (see tiles $M$, $M'$, and $\widetilde{M}$ in Figure~\ref{hat}).  Then 
 clearly  since  $\Delta(\widetilde{T})=0$,  $\Delta(\widetilde{T})\leq \Delta(T)$.  Since $T$ did not contain an endpoint of an arc of the complement, it was either type I or type II.  It cannot also be type II and contain a crossing or $\Delta(T)>0$, so we see that if $T$ was type II then $Cr(T)=0$ and $Cr(\widetilde{T}) = Cr(T)$.  If instead $T$ was type I and an arc of some $a_i$ passed through it, then again 
we must have $Cr(T)=0$ giving
 $Cr(\widetilde{T}) = Cr(T)$.   The 
 same problem as before, however, does need to be addressed. If $T$ was type I and was disjoint from the complement
then it might have had $Cr(T)=1$ and we have lost one crossing with the substitution, for example $Cr(\widetilde{M})=Cr(M)-1$.  Again, though,
Lemma~\ref{addedcrossings} together with the step we already took of replacing all the central tiles with 3-crossing tiles ensures that  these possible deficits are not a problem.

We are now done with all of our substitutions and have completed the definition of $\widetilde{L}$. Lemma~\ref{addedcrossings} states that the substitution we made on the central tiles adds at least as many crossings in as were lost in the two problem cases above ensuring $Cr(\widetilde{L})  \geq Cr(L)$.
 
 \bigskip

 By assumption for $r \geq 4$ $Cr(L)\geq 9r^2-30r+25 = s_r -(3r-3)$ where $s_r$ is the number of crossings in a saturated mosaic.  If $T$ is a tile of $L$ containing an arc of the complement $a_i$
then we know $Cr(T) \leq 1$ since $T$ contains at most two arcs of $L$.  Thus in this case $T$ has two fewer crossings than a saturated interior tile contains. Given the lower bound on $Cr(L)$, this implies that  the arcs of the complement can intersect at most $\frac{3r-3}{2}$ tiles.  When $r \geq 4$ this means that no $a_i$ is long enough to connect opposite sides of the penultimate corona since any such arc is of length at least $2r-3$ (which is greater than $\frac{3r-3}{2}$).
Thus any $a_i$ can span portions of one, two, or three sides of the penultimate corona, but not more.

We were concerned above with losing crossings on tiles that are type I in $L$, type II in $L'$, and are disjoint from the complement.  Any tile of this type must be outside of one of the arcs of the complement or it would be the same type in $L$ as it was in $L'$.   
 Let $B$ be the set of type II tiles of $L'$ for which the corresponding tile of $L$  was both outside of some arc of the complement of $L$ and disjoint from the complement.  Clearly $B$ includes all of the tiles on which we may have lost crossings during our substitutions and possibly some on which we did not lose crossings.
  
\begin{lemma}
If $B$ contains $t$  tiles  then the arcs of the complement of $L$, $\{ a_1,  a_2, \dots, a_n\}$, intersect at least $t$ central tiles.
\label{addedcrossings}
\end{lemma}

 \begin{proof}
 Note that we do not need the lemma when $r=3$ since we are able to skip this step, but it is trivially true in that case anyway because the penultimate corona of a 3-mosaic consists of exactly six corner tiles and so there are no type I or type II tiles which implies that in that case $B = \emptyset$ and so $t=0$.  We now proceed to the more interesting case of $r \geq 4$.

We start with the set of arcs of the complement of $L, A=\{ a_1,  a_2, \dots, a_n\}$ and build a new, related set of arcs $D= \{ d_1,  d_2, \dots, d_m\}, m \leq n$ which are similar, but have some nice properties explained below.  The proof works independent of any loops of the complement so we ignore them.  We describe the process here and show an example of turning $A$ into $D$ in Figure~\ref{fig:arcs}.  The first step in forming $D$ from $A$ is to take a subset of $A$ which we call $E$ which is the smallest subset such that each tile that was outside an arc of $A$ will also be outside an arc of $E$.  This gets rid of any nested arcs. See the left two pictures in Figure~\ref{fig:arcs} for an example of forming $E$ from $A$.

 Now we still might have two arcs say $a_i$ and $a_j$ which intersect the same central tile
as modeled by $a_2$ and $a_3$ in Figure~\ref{fig:arcs}.   If so delete the portion of the $a_i$ and $a_j$ in the shared tile and reconnect within the tile taking the connect sum of the two arcs so that now we have two new arcs $a_i'$ and $a_j'$ which intersect the same central tiles as $a_i$ and $a_j$ did, but with one of the new arcs outside of the other (as an example of this formation, see arcs $a_2'$ and $a_3'$  in Figure~\ref{fig:arcs}).  Note forming $a_i' \cup a_j'$ from $a_i \cup a_j$ is equivalent to simply swapping out one of the non-uniquely defined tiles of the complement from the top row of Figure~\ref{fig:comp} for the other choice. Without loss of generality say $a_j'$ is outside of $a_i'$.  Now every tile that is outside of $a_i$ or $a_j$ is also outside of $a_i'$.
Let the set $F$ be formed from $E$ by deleting $a_i$ and $a_j$ and adding in $a_i'$.  
Take subset $G$ of $F$ that is the smallest subset of $F$ such that each tile that was outside an arc of $F$ will also be outside an arc of $G$ again getting rid of any nested arcs.
Note that the set of central tiles that $G$ intersects is a subset of the central tiles that $A$ intersected.  Repeat this process until we have a set of arcs which intersect no common central tiles and which are not nested.  Call this new set of arcs  $D$ and for convenience rename the arcs $\{ d_1,  d_2, \dots, d_m\}, m \leq n$.

 Now if we can prove the lemma for the arcs of $D$ then clearly it is true for $A$ since
every tile that is outside an arc of $A$ is also outside an arc of $D$ and  $A$ intersects at least as many central tiles as $D$ does.

\begin{figure}[!ht]
   \centering
   \setlength{\unitlength}{0.1\textwidth}
   \begin{picture}(10,2.6)
     \put(0,0){\includegraphics[width=.25\textwidth]{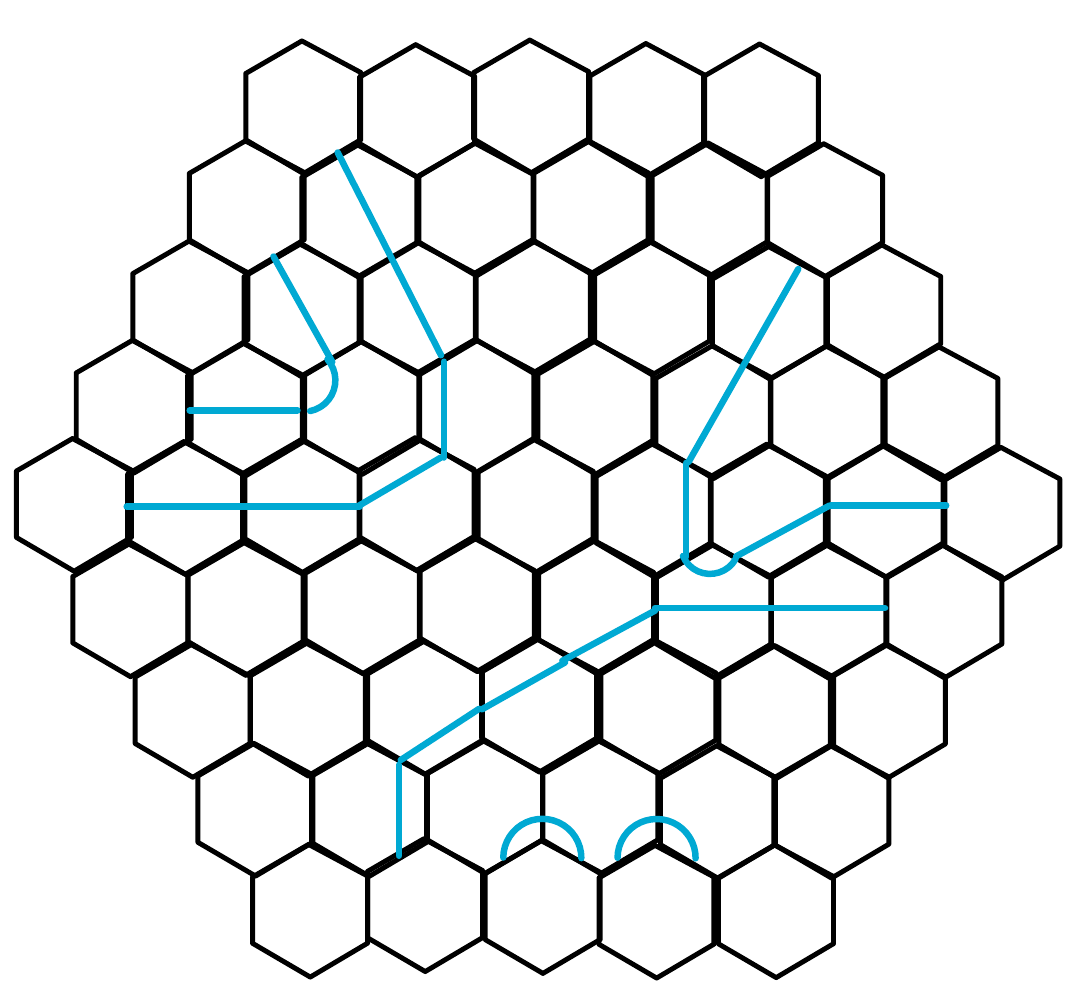}} 
     \put(2.5,0.01){\includegraphics[width=.25\textwidth]{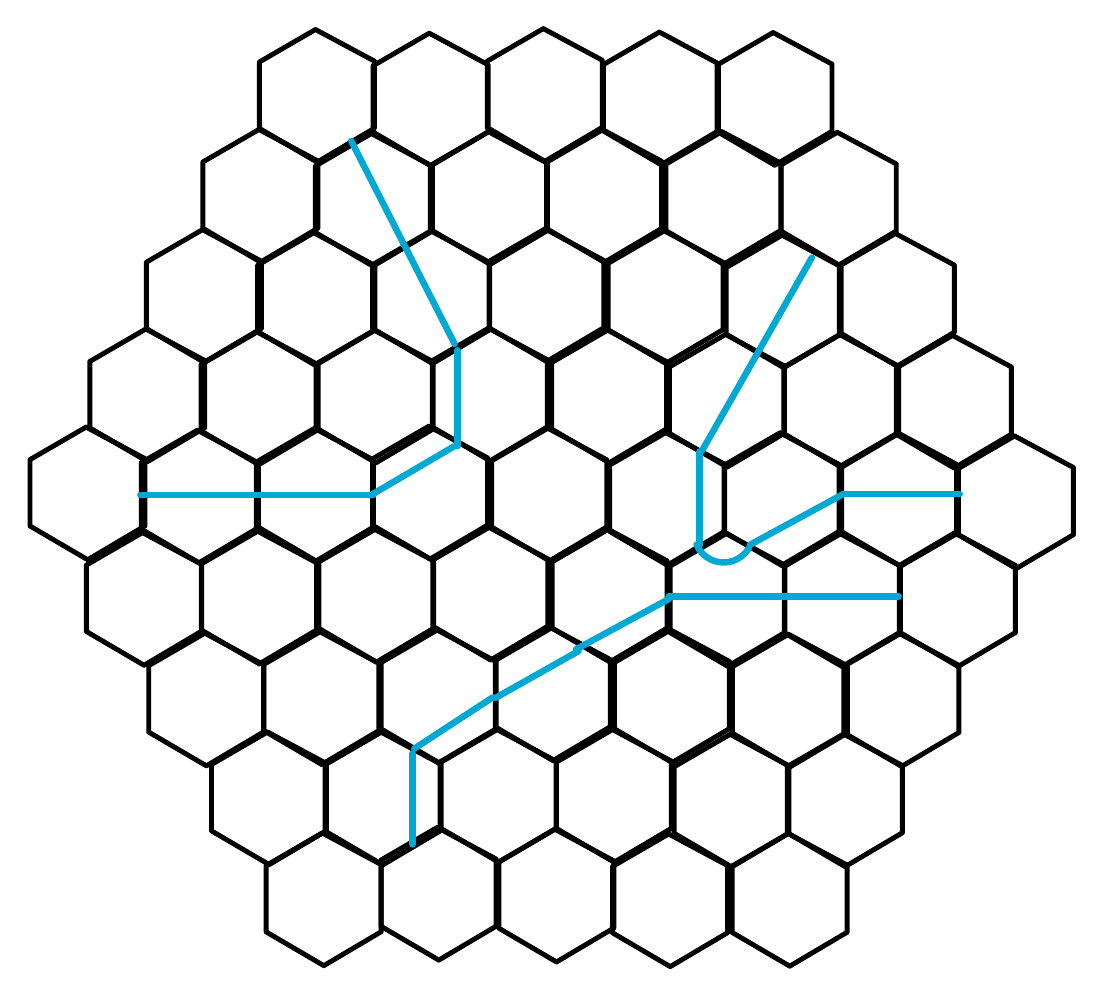}} 
     \put(5,-.03){\includegraphics[width=.25\textwidth]{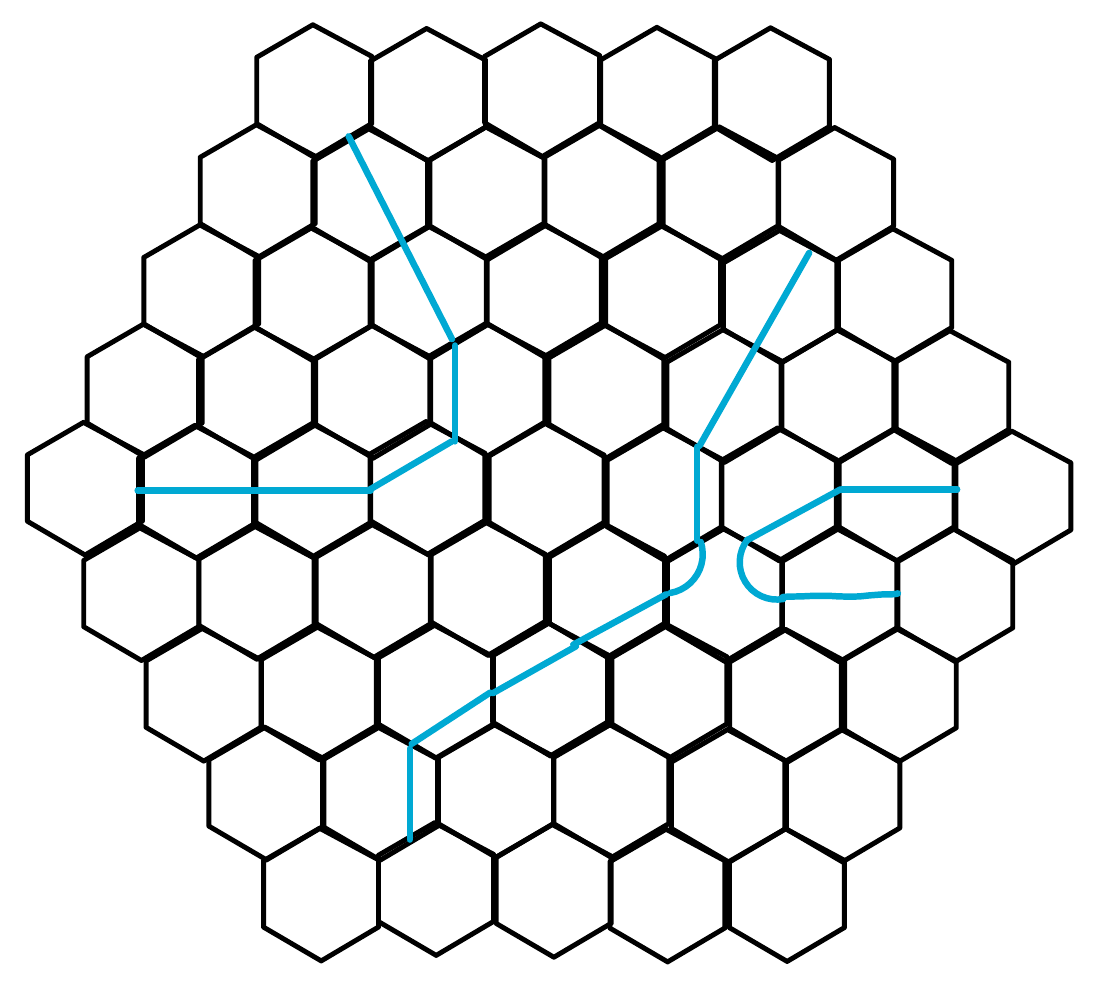}} 
     \put(7.5,-0.07){\includegraphics[width=.25\textwidth]{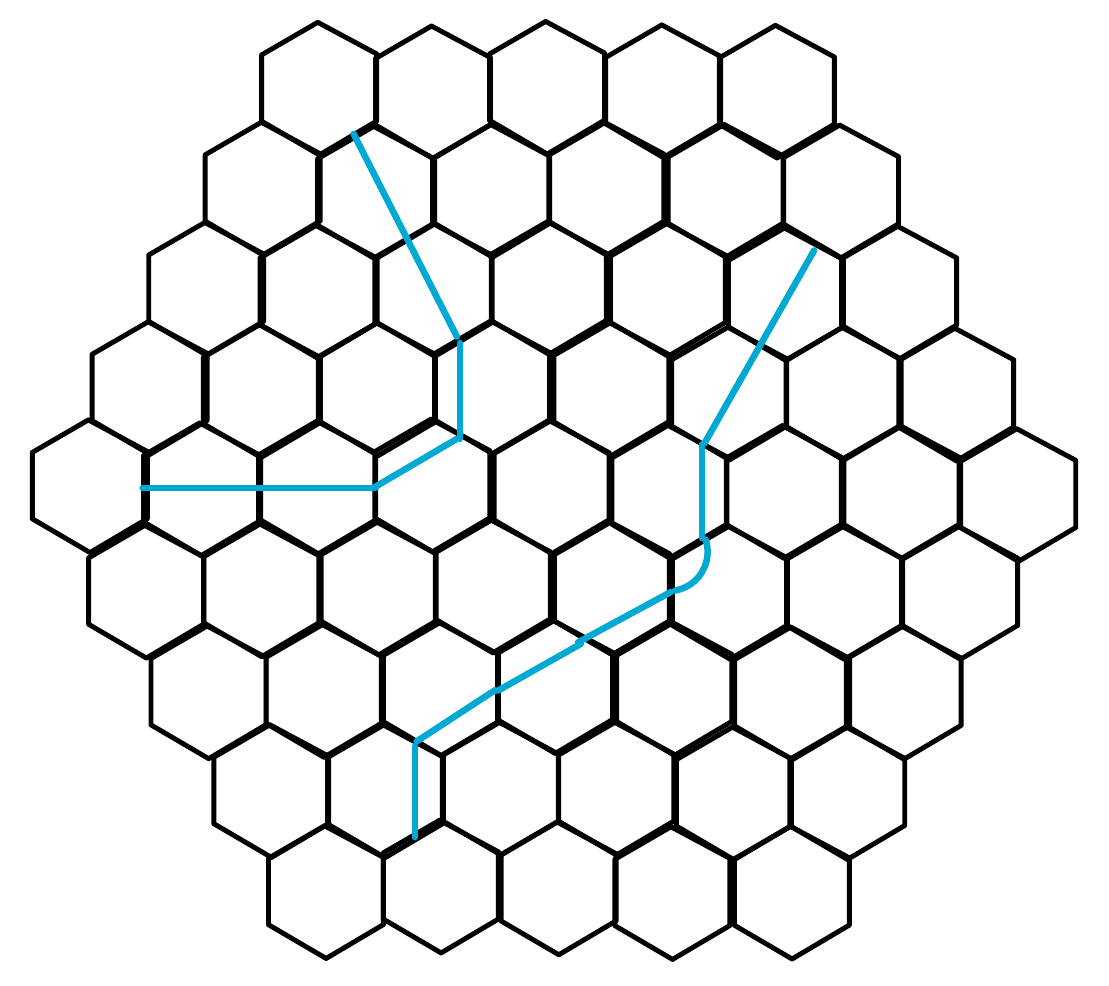}}
     \put(0.07,1.1){$a_1$}
     \put(0.85,0.2){$a_2$}
     \put(2.18,1.1){$a_3$}
     \put(0.19,1.33){$a_4$}
     \put(1.125,0.2){$a_5$}
     \put(1.4,0.2){$a_6$}
     \put(2.6,1.1){$a_1$}
     \put(3.41,0.2){$a_2$}
     \put(4.7,1.05){$a_3$}
     \put(5.08,1.02){$a_1'$}
     \put(5.88,0.12){$a_2'$}
     \put(7.17,1.01){$a_3'$}
     \put(7.6,1.00){$d_1$}
     \put(8.38,0.09){$d_2$}
   \end{picture}
\caption{In the first mosaic, on the left, we have the set of arcs $A = \{a_1, a_2, \dots a_6 \}$, In the second we take the subset $E = \{a_1, a_2, a_3 \}$ to get rid of the nested arcs in $A$. In the third we change $a_2$ and $a_3$ which meet in a central tile into $a_2'$ and $a_3'$.  Finally on the left we restrict to our final set of arcs $D$ and rename them $d_1$ and $d_2$. $D$ has the desirable properties that it has no nested arcs, every tile that was outside of an arc of $A$ is outside of an arc of $D$, and no two arcs of $D$ intersect the same central tile.}
\label{fig:arcs}
\end{figure}

 While a given $d_j$ may enter some of the penultimate tiles besides the ones containing
its endpoints, like $a_2$ does for $L$ in Figure~\ref{hat}, it suffices to prove the lemma for each of the sub-arcs of  $\{ d_1,  d_2, \dots, d_m\}$ that meet the penultimate corona only on the sub-arc's endpoints.   Call such a sub-arc $d'$.  If both endpoints of $d'$ are on the same side of the penultimate corona then the lemma is clearly satisfied as it must hit at least as many central tiles as the number of tiles it skips in the corona.  
The same is true  if the endpoints of $d'$ are on adjacent sides of the penultimate corona.  Since $d_i$ is not long enough to connect opposite sides of the corona, we are only left with the case where $d'$ has endpoints on two sides, say $h_1$ and $h_3$, but misses side $h_2$ lying between them completely.  

Note each side of the penultimate corona of a mosaic contains $r-3$ tiles which are not corner tiles.  Recall that the set $B$ is a subset of the type II tiles in $L'$.   We claim there are  no more than $r-3$ type II tiles of $L'$ outside $d_i$ (and thus $d'$ must obey the bound, too).  
To connect non-adjacent sides without having an endpoint on a corner tile, we know the length of $d'$ must be at least $r$ in which case it hits at least $r-2$ central tiles.  Either $h_1$ and $h_3$ contain only type II tiles and $h_2$ only contains type I or vice versa.  If $h_2$ contains the type II tiles then since it has $r-3$ non-corner tiles there are exactly $r-3$ type II tiles outside  of $d'$.  
If on the other hand $h_2$ contains type I tiles and the non-corner tiles of $h_1$ and $h_3$ are the type II tiles then observe that if $d'$ is length $r+k$  then it can include at most $k$ non-corner tiles of $h_1$ on its outside and the same for $h_3$.  We know the length of $d_i$ (and thus $d'$) can exceed $r$ by at most $\frac{3r-3}{2} -r = \frac{r-3}{2}$.  
Such an arc will therefore have at most $\frac{r-3}{2}$ non-corner tiles outside it from $h_1$ and the same for $h_3$  for a maximum total of $r-3$  type II tiles.  
Since in this case $d'$ is disjoint from an entire side of the penultimate corona (on its outside) it  must pass through at least $r-2$ central tiles which is, of course, greater than $r-3$ proving the final case of the lemma.
 \end{proof}
 
\bigskip

\emph{Conclusion of Case 1:} 
 We now know that $Cr(\widetilde{L}) \geq Cr(L)$, we also know that 
 for each penultimate tile $\widetilde{T} \subset \widetilde{L}$ and corresponding tile $T \subset L$ that $\Delta(\widetilde{T}) \leq \Delta(T)$.
  It is theoretically possible for central tiles to contribute to $\Delta(\widetilde{L})$ if we make a foolish substitution.  Central tiles, for example, 
contribute to $\Delta(L')$ in Figure~\ref{hat}.  
We, however, carefully prevented this in our construction of $\widetilde{L}$
because each time we substituted we rotated the replacement tile to ensure no edge of the dual graph
could pass through it from the exterior vertex to the central tiles without hitting one of the edges of $\widetilde{L}$ and our substitutions in the penultimate corona were also carefully chosen so that they never created any any nugatory crossings on the exterior of the mosaic.
This implies that $\Delta(\widetilde{L})$ is exactly the sum of the values from the penultimate tiles.   Thus we now have shown that  $\Delta(\widetilde{L}) \leq \Delta(L)$ and $Cr(\widetilde{L}) \geq Cr(L)$, but $\widetilde{L}$ has no arcs or loops in its complement.  
So now we know that for $r \geq 4$ if $L$ exists then $\widetilde{L}$ exists with $\Delta(\widetilde{L}) \leq 7r-13$ and $Cr(\widetilde{L}) \geq 9r^2-30r+25$ ($\Delta(\widetilde{L}) \leq 7$ and $Cr(\widetilde{L}) \geq 15$ for $r=3$),  and $\widetilde{L}$ has no nugatory crossings on its exterior and no arcs or loops in its complement so we proceed to case 2.

 \bigskip

{\bf Case 2: $\widetilde{L}$ has no arcs or loops in its complement. }

\bigskip

Since there are no arcs in the complement, the boundary tiles of the mosaic must use all possible connection points.   There are, of course, two ways to connect up the boundary tiles in a standard mosaic, and in this situation the two ways are the same up to rotation of the entire mosaic by $\pi/3$ so we may assume $\widetilde{L}$ and $L_r$ have identical boundary tiles and that $\widetilde{L}$ can be formed from $L_r$ by replacing some portion of the interior tiles. We used $T$ to refer to tiles of $L$ earlier and will use subscripts to refer to multiple tiles from the same mosaic later, so we will refer to an individual tile belonging to $L_r$ as $T^r$ and the corresponding tile in $\widetilde{L}$ as $\widetilde{T}$.
In particular, we will focus on the substitution when $T^r$ is one of the penultimate tiles and how $\Delta(\widetilde{T})$ compares to $\Delta(T^r)$. 
We argue that for each tile in this corona if  $\Delta(T^r)- \Delta(\widetilde{T})= t$ then $Cr(T^r)-Cr(\widetilde{T}) \geq 1.5t$.  Call this property \emph{the tile substitution inequality}.  While as before there theoretically could be some central tiles that contribute to $\Delta(\widetilde{L})$
the inequality above will ensure that for $r\geq 4$ if the penultimate tiles have not already driven $\Delta(\widetilde{L})$ above $7r-13$ then they must be missing so many crossings that $Cr(\widetilde{L})$ must be below $9r^2-30r+25$ (and for $r=3$ that if $\Delta(\widetilde{L})$ is pushed above $7$ then $Cr(\widetilde{L})$ falls below $15$),  a contradiction.

As in any saturated mosaic ($r \geq 3$), the penultimate corona
of $\widetilde{L}$ (and of $L_r$) has 6 corner tiles, $3(r-3)=3r-9$  type I tiles and $3r-9$  type II tiles.  If $T^r$ in the penultimate corona is a corner tile for $L_r$ 
then $\Delta(T^r)=2$, if $T^r$ is a type I tile then $\Delta(T^r)=1$, and  if $T^r$ is a type II tile then $\Delta(T^r)=2$.  Claim~\ref{claim:lrextdeg} shows that these contributions add to $9r-15$.  We have assumed  $\Delta(\widetilde{L}) \leq 7r-13$.     Thus the replacements in the penultimate corona that form $\widetilde{L}$ must  lower $\Delta$ by at least $2r-2$ for $r \geq 4$ (or by 5 for $r=3$)
without lowering the total number of crossings too much. We now inspect the three possible types of tiles.

 \begin{itemize}
 
\item \emph{A) $T^r$ is a corner tile:} For corner tiles $\Delta(T^r)=2$ so if $\Delta(\widetilde{T}) < \Delta(T^r)$ then $\Delta(\widetilde{T}) = 1$ or 0.
If $\Delta(\widetilde{T})=1$ then $Cr(\widetilde{T})=1$ so in this case
$\Delta(T^r)- \Delta(\widetilde{T})= 1$ and $Cr(T^r)-Cr(\widetilde{T}) =2$. On the other hand, if $\Delta(\widetilde{T})=0$ then $Cr(\widetilde{T})=0$ so in this case $\Delta(T^r)- \Delta(\widetilde{T})= 2$ and  $Cr(T^r)-Cr(\widetilde{T}) =3$ so in either case the tile substitution inequality is satisfied.

\item \emph{B) $T^r$ is a type I tile:}
In this case if $\Delta(\widetilde{T}) < \Delta(T^r)$ then $\Delta(\widetilde{T})=0$ since $\Delta(T^r)=1$, but  if $\Delta(\widetilde{T})=0$ then $Cr(\widetilde{T}) \leq 1$ since every type I  two or three crossing tile contributes at least one to $\Delta$.  Thus in this case in order to lower $\Delta(T^r)$ by one we must lower $Cr(T^r)$ by two and as before the tile substitution inequality is satisfied.

\item \emph{C) $T$ is a type II tile:}
For these tiles if $\Delta(\widetilde{T})=1$ then $Cr(\widetilde{T})=1$ so again we must lose two crossings to lower $\Delta(T^r)$ by one, and finally
for type II tiles if $\Delta(\widetilde{T})=0$ then $Cr(\widetilde{T})=0$ so we must lose three crossings to lower  $\Delta(T^r)$ by two so in this case, too,  the tile substitution inequality is satisfied.

 \end{itemize}

Thus in every case lowering $\Delta(T^r)$ by  $t$   lowered the total number of crossings by at least $1.5t$.  We constructed $\widetilde{L}$ so that  for $r\geq 4$, $\Delta(L_r) - \Delta(\widetilde{L}) \geq 2r-2$ 
and if we restrict to the sums over the tiles in the penultimate corona then $\Sigma(\Delta(T^r_i)) - \Sigma(\Delta(\widetilde{T_i})) \geq 2r-2$ since for $L_r$, $\Delta(L_r)=  \Sigma(\Delta(T^r_i))$.  This, however, implies that if
$\Delta(\widetilde{L}) \leq 7r-13$ then $\Sigma(Cr(T^r_i))- \Sigma(Cr(\widetilde{T_i})) \geq  \frac{3(2r-2)}{2}= 3r-3$
and thus $Cr(L_r)- Cr(\widetilde{L}) \geq  3r-3$. Thus $\widetilde{L}$ has at most $9r^2-30r+24$ crossings.  
 For  $r=3$ to lower $\Delta$ by five we lose at least $7.5$ crossings relative to $L_3$, showing $Cr(\widetilde{L}) \leq 14.5$ and thus is clearly below 15 as desired.
 
 This contradicts the assumption that for $r\geq 4$, $\widetilde{L}$ has crossing number $9r^2-30r+25$ or more which in turn  proves that there cannot be a link $L$ on hexagonal $r$-mosaic with $\Delta(L) \leq 7r-13$ and $Cr(L) \geq 9r^2-30r+25$ finalizing the proof of the theorem for $r \geq 4$.  Similarly for $r=3$ the argument leads to the conclusion that if $\Delta(L) \leq 7$ then $Cr(L) \leq 14$ completing that case, too.

\end{proof}

\begin{cor}
For $r \geq 4$ any hexagonal $r$-mosaic containing $K_r$ 
must have exterior degree strictly greater than $7r-13$ and for $r=3$ any hexagonal $3$-mosaic containing $K_3$ 
must have exterior degree strictly greater than $7$.
\label{cor:krext}
\end{cor}

\begin{proof}
We showed in Claim~\ref{claim:krarcn} that the crossing number of $K_r$ is
$9r^2-30r+25$ for $r \geq 4$ and the crossing number of $K_3$ is 15, so this follows immediately from the previous theorem.
\end{proof}

 We will next show that in any reduced alternating form $K_r$ has exterior degree at most $7r-13$ for $r \geq 4$ and that $K_3$ has exterior degree at most 7 which 
combined with Corollary~\ref{cor:krext} 
  will show that it cannot be embedded in reduced alternating form on an $r$-mosaic.

We know our standard alternating embedding of $K_r$ has exterior degree $7r-13$ for $r \geq 4$ (or 7 for $K_3$).  We need to show that no other reduced alternating embedding has exterior degree above that value.  Since our diagram is indeed reduced and alternating we know by the Tait flyping conjecture that if we think of it on $S^2$ then every other reduced and alternating diagram of the knot can be obtained from this one by flypes \cite{MT}.

\begin{figure}[tpb]
\centering
\includegraphics[width=0.65\textwidth]{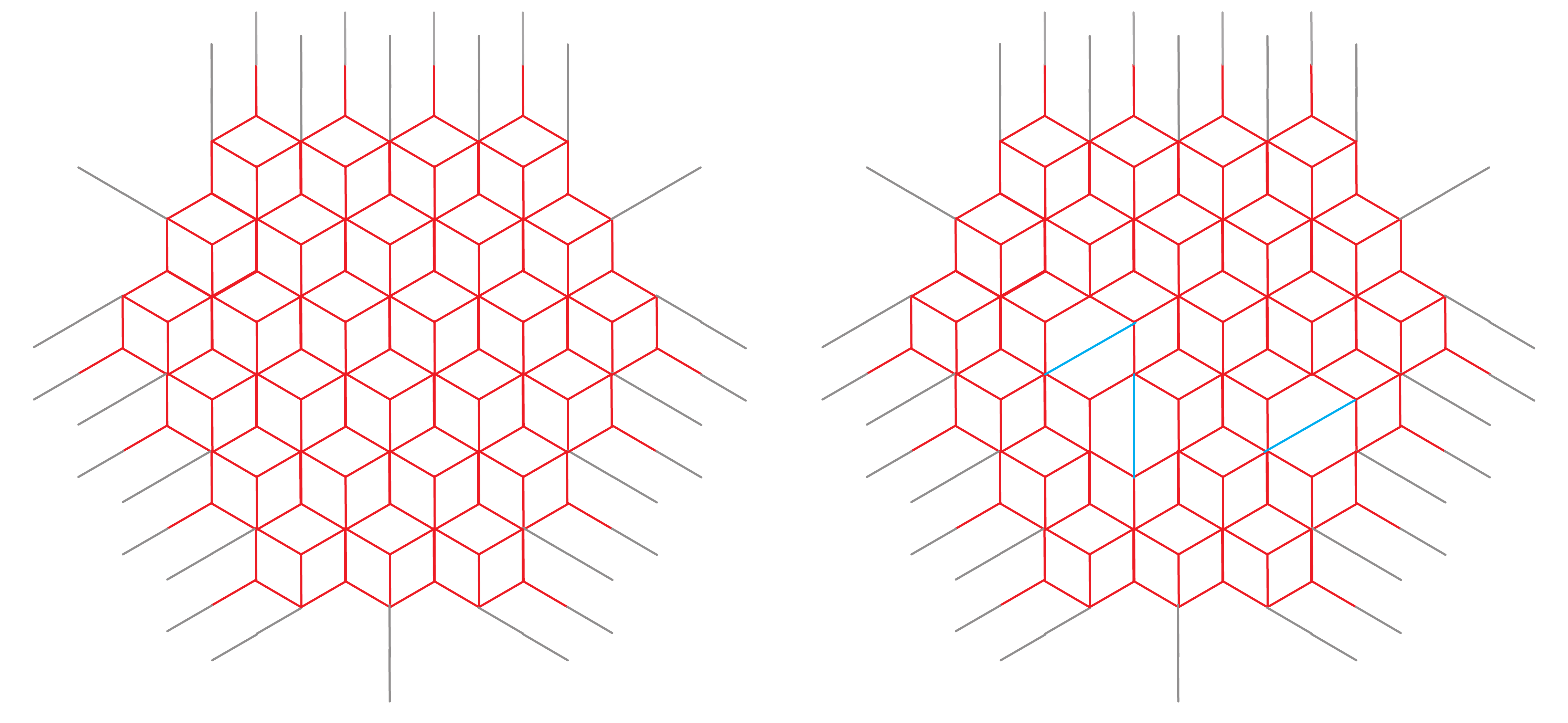}
\caption{On the left we see (most of) the dual graph for $L_5$ in its standard alternating diagram.  The exterior vertex
is not drawn.  All the gray edges would connect up to the exterior vertex. On the right we see the dual graph for $A_5$ (as well as $K_5$ in its non-alternating diagram) formed from the graph on the left by deleting sets of three edges near the center of the graph and replacing them with single edges.  The replacement edges are in blue to make them easier to identify.}
\label{fig:l5dual}
\label{fig:a5dual}
\end{figure}

\begin{claim}
The highest exterior degree of any reduced alternating embedding of $K_r, r \geq 4$ is $7r-13$ (7 for $K_3$), and the highest exterior degree is achieved in the standard alternating embedding.  
\label{claim:highest}
\end{claim}

\begin{proof}

As we've seen the construction of $K_3$ was slightly different than the construction of $K_r$ for $r \geq 4$.  As a result the dual graph and the flypes behave differently for $K_3$ than the other knots.   
We inspect $K_3$ first and then all the other cases are argued together.
The standard alternating diagram of  $K_3$ seen in Figure~\ref{fig:k3dual} has exterior degree is 7 (see vertex $v_1$ in the figure).   Searching the dual graph we see that there are a relatively small number of possible flypes. If we try them all we see that none of them drive up the degree of $v_1$.  The most interesting thing that happens is that the flype corresponding to the nested 4-cycles with outer 4-cycle $v_1, v_7, v_4, v_9$ lowers the degree of $v_1$ by 2 and raises the degree of $v_4$ to 7 (it will turn out that in no other $K_r$ do the flypes ever cause the degree of an interior vertex to match that of the initial exterior vertex).
The key issue for $K_3$ is that the maximal degree of any vertex in any reduced alternating diagram is 7 so certainly the exterior degree never exceeds 7 in any embedding showing $K_3$ satisfies the claim.

For $r \geq 4$ we argue that even after any combination of flypes $v_1$ has higher degree than any other vertex
and it never has higher degree than in the initial embedding  so the argument is actually simpler than for $K_3$. 
To find flypes, we search the dual graph for nested 4-cycles like we saw in Figure~\ref{fig:flype}.  
The dual graph for $L_r$ is easy to visualize based on the dual graph for $L_5$ shown on the left in Figure~\ref{fig:l5dual}.  
The dual graph for $K_r$ before the isotopy that makes it alternating is the same as the dual graph for $A_r$, which is obtained from the dual graph for $L_r$ by replacing three edges meeting at a vertex of degree three with a single edge as shown for $L_5$ and $K_5$ in the progression from the figure on the left to the one on the right in Figure~\ref{fig:l5dual}.  
In both of the graphs in the figure every non-empty 4-cycle (and thus all nested 4-cycles) contains $v_1$ and because the substitutions are towards the center of the graph all non-empty 4-cycles are disjoint from the new edges and their vertices.  
The isotopy to make $K_r$ alternating  breaks the symmetry of the dual graph and reduces the exterior degree as shown for $K_5$ in Figure~\ref{fig:k5flype}.

\begin{figure}[tpb]
\centering
\includegraphics[width=0.24\textwidth]{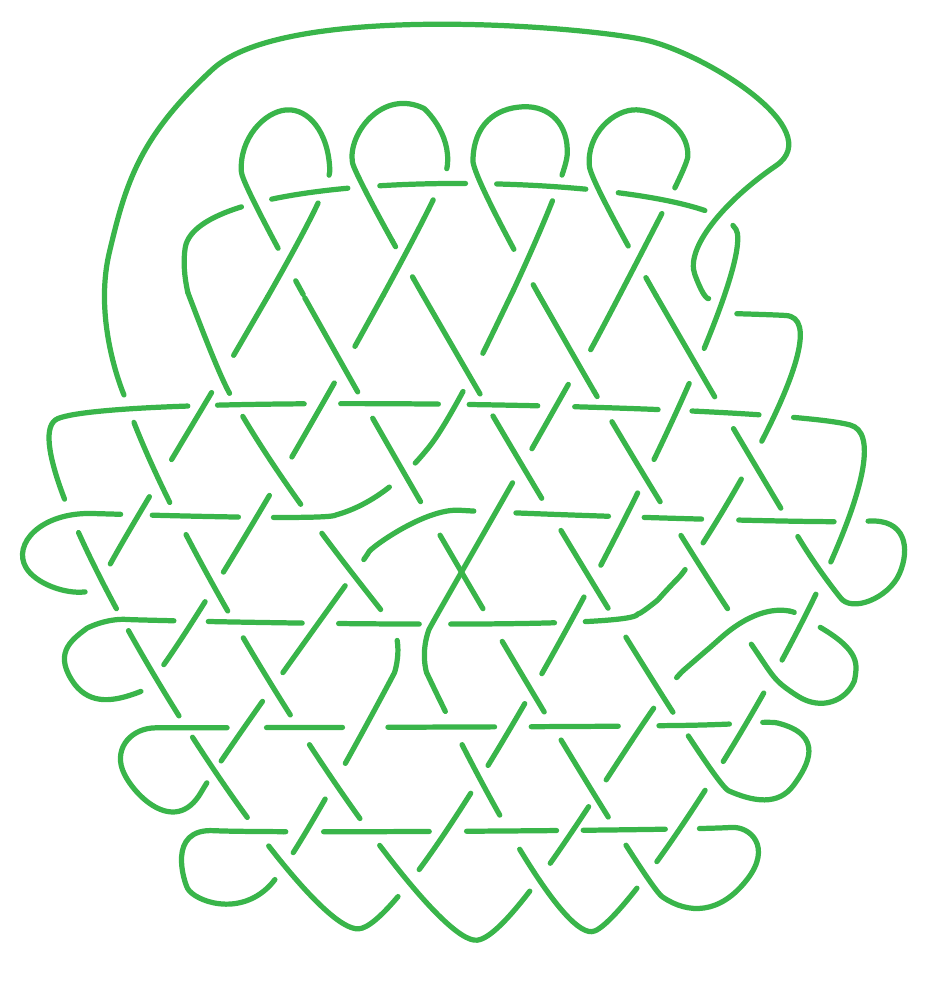}
\includegraphics[width=0.24\textwidth]{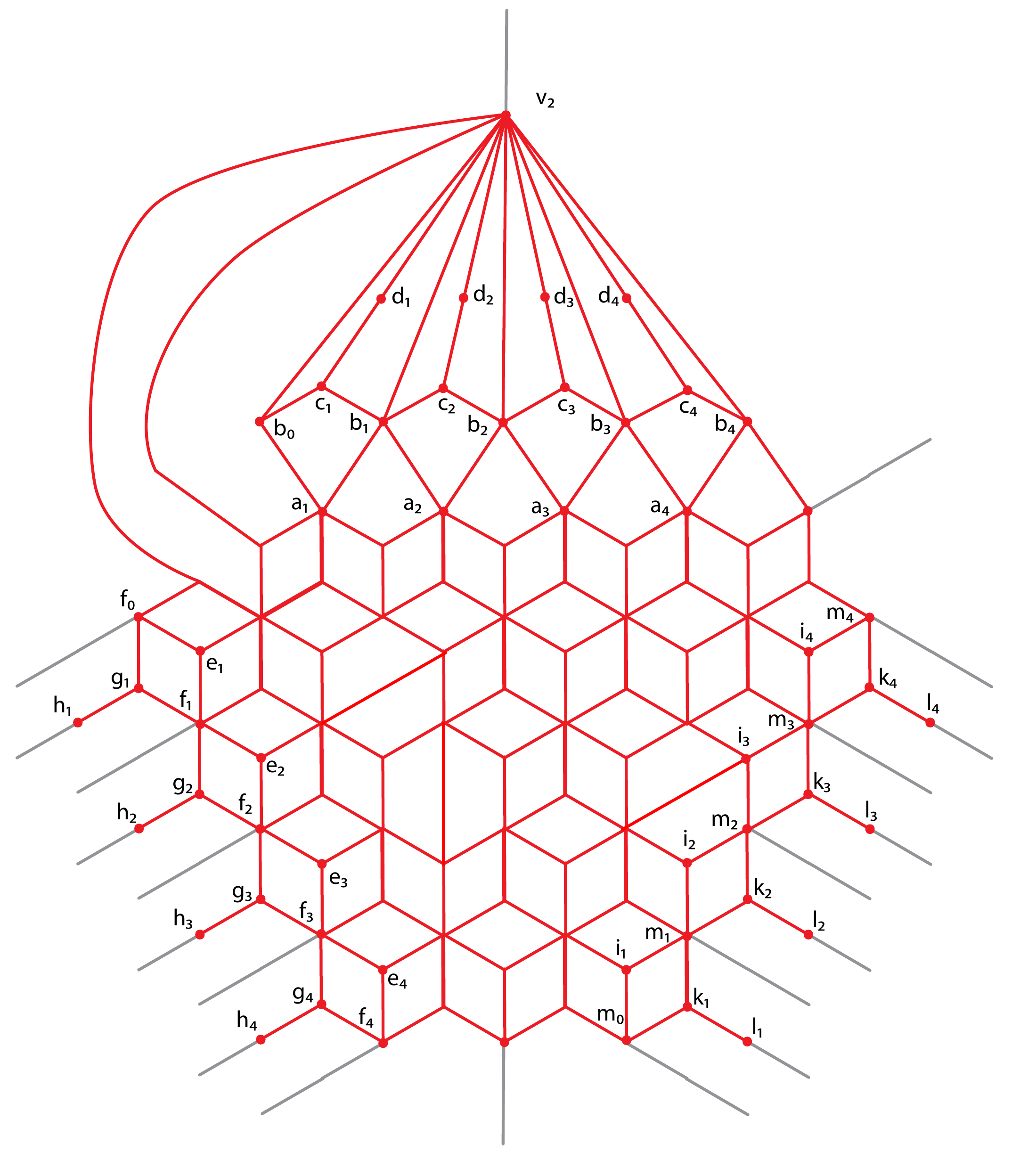}
\includegraphics[width=0.24\textwidth]{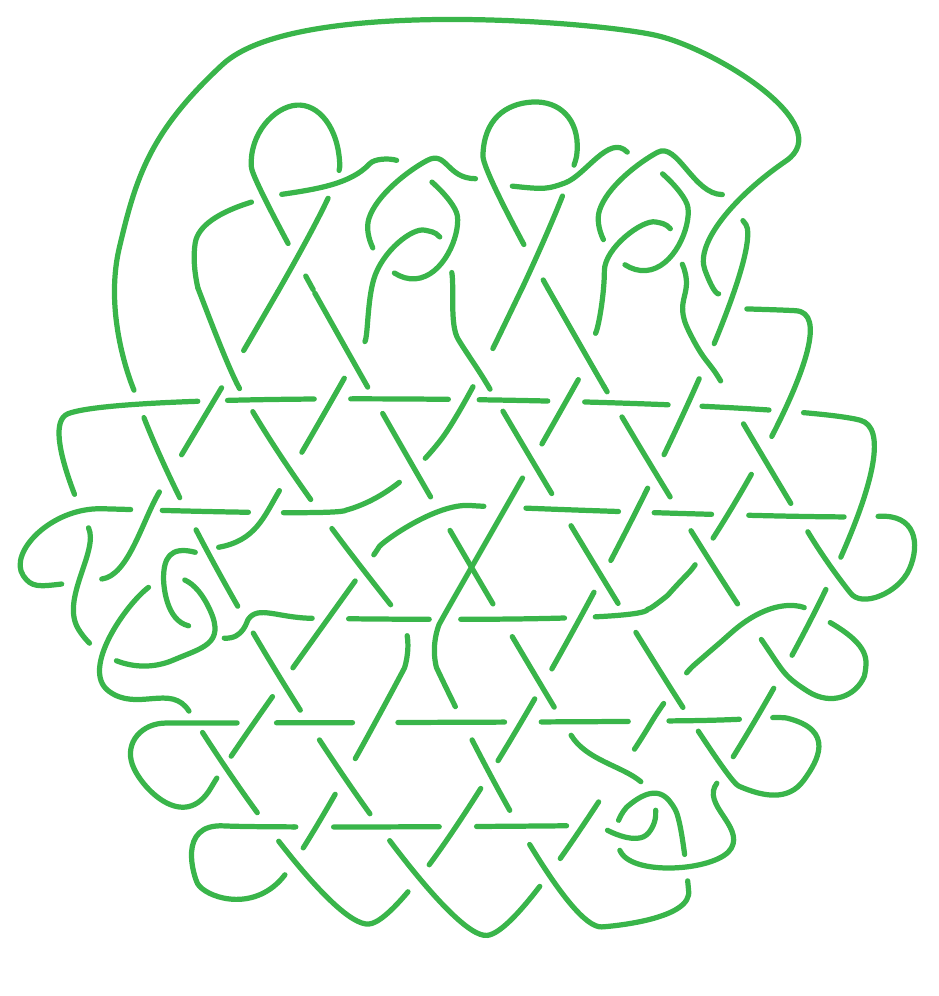}
\includegraphics[width=0.24\textwidth]{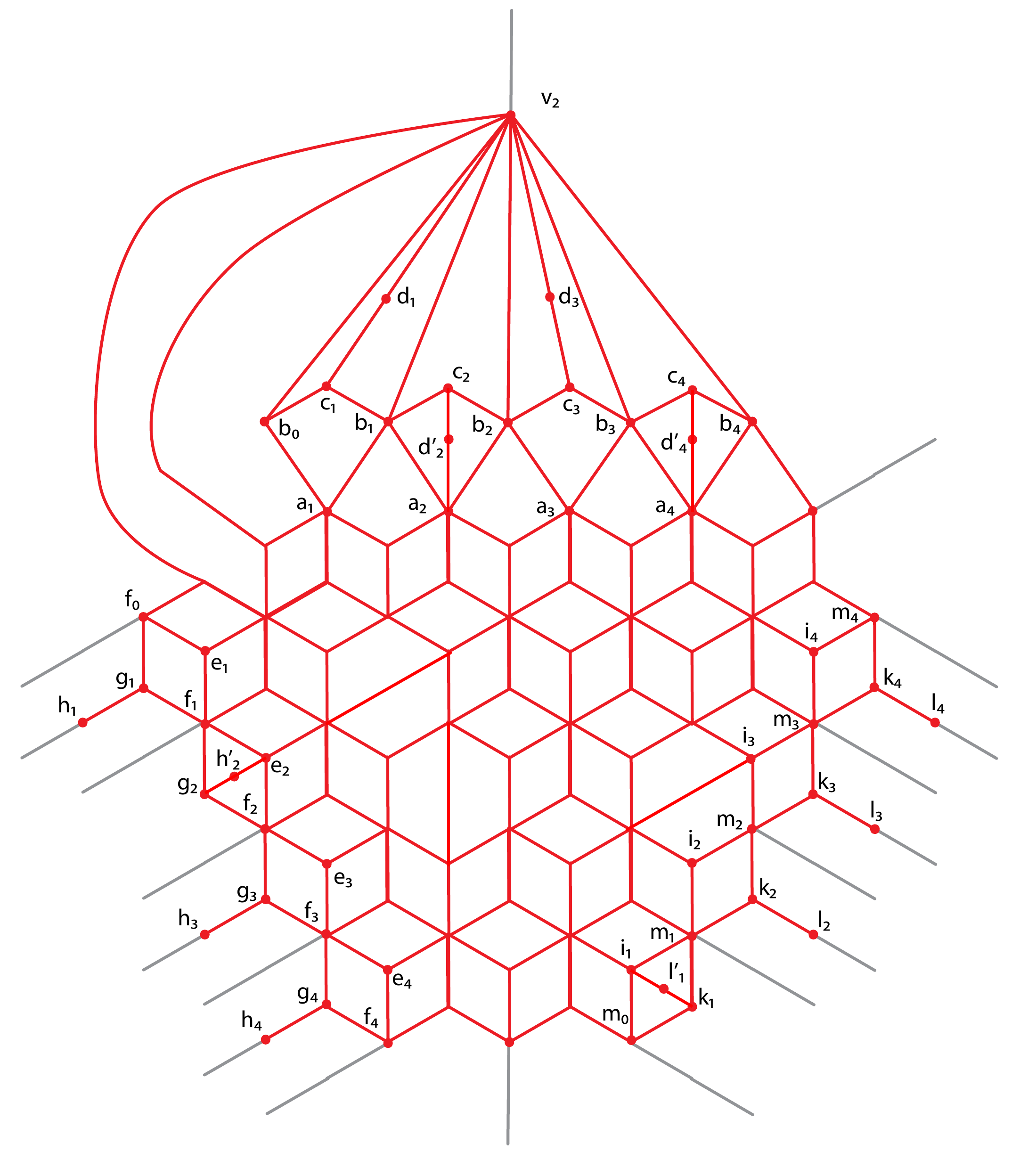}
\caption{The knot $K_5$ in our standard alternating diagram and its dual graph (aside from $v_1$) are shown on the left.  Four flypes are then performed to get to the picture on the right.  Two flypes contain $v_1$ corresponding to the outer 4-cycles $v_1, f_1, e_2, f_2$, $v_1, m_0, i_1, m_1$, 
and two contain $v_2$ corresponding to the outer 4-cycles $v_2, b_1, a_2, b_2$, and $v_2, b_3, a_4, b_4$.
Note that the first type of flype lowers the exterior degree by one each time and the second type of flype lowers the degree of $v_2$ by one each time.}
\label{fig:k5flype}
\end{figure}

  Call the dual graph to $K_r$ in the standard alternating diagram $D_r$.  
As usual call the exterior vertex $v_1$ and next highest degree vertex (which is separated from $v_1$ by the edge that was pulled up in the isotopy to make the diagram alternating) $v_2$.  

We first find all the nested 4-cycles in the graph which contain $v_1$.  It is easy to see that there are
$2r-2$ of these.  Recall that $v_1$ is not drawn in Figure~\ref{fig:l5dual} and all gray edges on the graph end at $v_1$.  Thus
on the bottom left side we see outer 4-cycles $(v_1, f_0, e_1, f_1)$,
$(v_1, f_1, e_2, f_2)$, $(v_1, f_2, e_3, f_3), \dots$, $(v_1, f_{r-2}, e_{r-1}, f_{r-1})$ and for the bottom right side $(v_1, m_0, i_1, m_1)$,
$(v_1, m_1, i_2, m_2)$, $(v_1, m_2, i_3, m_3), \dots$, $(v_1, m_{r-2}, i_{r-1}, m_{r-1})$. The cycles can be seen in  Figure~\ref{fig:k5flype} for the case of $r=5$, but the patterns continue as $r$ grows larger.  Examples of the possible flypes are then displayed on the right side of the figure showing their impact both for the dual graph above and the knot below.

$D_r - v_1$ has $r-1$ sets of nested 4-cycles containing $v_2$.  The vertices are labeled in Figure~\ref{fig:k5flype} to help see the nested 4-cycles in $K_5$, and also the general pattern for $K_r$.
The nested 4-cycles containing $v_2$ have outer 4-cycle $(v_2, b_0, a_1, b_1)$,
$(v_2, b_1, a_2, b_2)$, $(v_2, b_2, a_3, b_3), \dots$, $(v_2, b_{r-2}, a_{r-1}, b_{r-1})$.
$D_r$ has these $r-1$ nested 4-cycles, plus the $2r-2$ more which contain $v_1$.  
It is easy to see that $D_r - (v_1 \cup v_2)$ contains only empty 4-cycles so we have already found all nested 4-cycles.

  Note that the flypes involving $v_1$ drop the degree of $v_1$ by one.  They each raise the degree of another vertex ($e_s$ or $i_s$ depending on which is contained in the outer 4-cycle) from 3 to 4.  The flypes involving $v_2$ do not change the degree of $v_1$, but lower the degree of $v_2$ by one and raise the degree of an $a_s$ from 5 to 6.  As you can see in the figure, the flypes essentially swap the two inner 4-cycles, deleting the two edges running through the 4-cycle that was not empty and adding 2 edges to the previously empty 4-cycle (the effect of a flype on a dual graph can be slightly more complicated than this in general as seen in $K_3$, for example, but the tangle that gets flipped is quite simple in $K_r, r \geq 4$ so the impact on the graph is also simple).

We now know that no flype we can apply to the standard alternating diagram of $K_r$ can increase the exterior degree or increase the degree of $v_2$ above its initial degree.  We also see that after a single flype the degree of any vertex other than $v_1$ or $v_2$ is at most 6.  Thus after a single flype no vertex can have degree higher than the initial degree of $v_1$.  We do, however, want to be certain that there isn't a sequence of flypes which put together could increase the degree of some vertex to a higher value than the initial exterior degree.

It is easy to see that any of the original flypes can be performed independently, but once a given flype is performed the only new flype introduced is the the inverse of the previous flype since the only new nested 4-cycle created is the inverse of the original.  Thus every alternating diagram of $K_r$ is obtained from the standard alternating diagram by doing a subset of the $3r-3$ flypes we found initially. 
All of these flypes either leave the degree of $v_1$ alone or decrease it. The degree of no other vertex ever exceeds the degree of $v_1$ so no vertex in any reduced alternating diagram of $K_r$ ever has higher degree than $v_1$ has in the standard alternating diagram (for $r\geq 4$). Thus the exterior degree of $K_r$ in every possible reduced, alternating diagram is at most $7r-13$.

\end{proof}

Corollary~\ref{cor:krext} combined with Claim~\ref{claim:highest} lead directly to the following theorem.

\begin{theorem}
For every $r \geq 3$, $K_r$ can be embedded on a hexagonal $r$-mosaic, but it cannot achieve its crossing number and be embedded on an $r$-mosaic at the same time.
\label{thm:fit}
\end{theorem}

\section{Open Questions And Conjectures}

Because hexagonal mosaics are relatively new there are a huge number of interesting open questions about them.  Related to Theorem~\ref{thm:fit} one might ask the following.

\begin{q}
What is the smallest crossing number knot which fits on a  hexagonal $r$-mosaic, but not while achieving its crossing number?
\end{q}

\begin{q}
Is there a knot that fits on a  hexagonal $r$-mosaic, but does not on fit on a  hexagonal $r+k$-mosaic while achieving its crossing number for $k \in \Z^+$?
\end{q}

Nearly any question that has been asked or answered about rectangular mosaics can also be asked here.  Here is one obvious example.
The Lomonaco-Kauffman Conjecture states that tame knot theory is equivalent to rectangular mosaic knot theory and one can extend the conjecture to hexagonal mosaics.  The Lomonaco-Kauffman conjecture was proven for rectangular mosaics by Kuriya and Shehab in \cite{KS}.

\begin{conj}
The Lomonaco-Kauffman Conjecture is true for hexagonal mosaics.
\end{conj}

Extending beyond these variants of existing theorems, there are still plenty of interesting questions.  For example, mosaics are a natural setting for looking at random knots.   Hexagonal mosaics build knots far more efficiently than rectangular mosaics so they are especially appealing in this context (see, for example, \cite{AAGRK}).  In the rectangular setting 4-mosaics contain at most 4 crossings.  On the other hand, hexagonal 4-mosaics have up to 57 crossings so while rectangular mosaics are a natural way to generate random knots hexagonal mosaics seem even better.  Petal diagrams of knots have been used in the past to generate random knots in \cite{hass1}
and \cite{hass2}, for example.  Because most links do not have a petal diagram the diagrams had to be adjusted when applied to random links.  No such adjustment is necessary for mosaics.  It would be interesting to see how hexagonal mosaics compare to petal diagrams for generating knots and links.

\begin{q}
What is the distribution of knots (or links) in hexagonal mosaics?  How does this compare to the distribution in rectangular mosaics or petal diagrams.
\end{q}

Since rectangular mosaics use only 5 tiles to generate all knots, it might be computationally helpful to restrict to fewer tiles in the hexagonal case.  Adams showed that every knot has a triple crossing diagram \cite{Adams2} and one can use this to show that every knot can be created as a hexagonal $r$-mosaic for some $r$ while restricting to only the three crossing tiles and the tiles with no crossings.  One can probably use even fewer tiles though.

\bigskip\bigskip

\bibliographystyle{amsplain}

  \bibliography{Hex}

\end{document}